\documentclass[9pt]{amsart}
\usepackage{epsfig}
\usepackage{amsmath, amssymb, euscript,  enumerate}
\usepackage{pxfonts}
\usepackage{color,wrapfig}
\newcommand{\R}{{\mathbb R}}

\newcommand{\ap}{\alpha}             
\newcommand{\bt}{\beta}
             \newcommand{\Gm}{\Gamma}
\newcommand{\dt}{\delta}             
\newcommand{\vep}{\varepsilon}

\newcommand{\ld}{\lambda}            \newcommand{\Ld}{\Lambda}
\newcommand{\sm}{\sigma}             
\newcommand{\vp}{\varphi}
\newcommand{\om}{\omega}             \newcommand{\Om}{\Omega}
\newcommand{\vr}{\varrho}            \newcommand{\iy}{\infty}
            
\newcommand{\f}{\frac}             \newcommand{\el}{\ell}

\newcommand{\fL}{{\mathfrak L}}

\newcommand{\fa}{{\mathfrak a}}
\newcommand{\fb}{{\mathfrak b}}

\newcommand{\fm}{{\mathfrak m}}

\newcommand{\BN}{{\mathbb N}}

\newcommand{\BR}{{\mathbb R}}

\newcommand{\cB}{{\mathcal B}}
\newcommand{\cC}{{\mathcal C}}

\newcommand{\cI}{{\mathcal I}}
\newcommand{\cJ}{{\mathcal J}}
\newcommand{\cK}{{\mathcal K}}
\newcommand{\cL}{{\mathcal L}}
\newcommand{\cM}{{\mathcal M}}

\newcommand{\cS}{{\mathcal S}}

\newcommand{\tQ}{{\widetilde{ Q}}}

\newcommand{\Qtil}{{\widetilde Q}}

\newcommand{\s}{\setminus}         \newcommand{\ep}{\epsilon}
\newcommand{\n}{\nabla}            \newcommand{\e}{\eta}
\newcommand{\pa}{\partial}        \newcommand{\fd}{\fallingdotseq}
       
    \newcommand{\ds}{\displaystyle}

 \newcommand{\pf }{\noindent{\it Proof. }}

\newcommand{\aee }{\text{\rm a.e.}} 
  \newcommand{\pv }{\text{\rm p.v.}}
  \newcommand{\dd }{\text{\rm d}}

\newcommand{\rB }{{\text{\rm B}}}   \newcommand{\rC }{{\text{\rm C}}}

\newtheorem{thm}[subsubsection]{Theorem}
\newtheorem{lemma}[subsubsection]{Lemma}
\newtheorem{cor}[subsubsection]{Corollary}
\newtheorem{remark}[subsubsection]{Remark}

\newtheorem{definition}[subsubsection]{Definition}

\numberwithin{equation}{subsection}

\title[fully nonlinear integro-differential operators]{ Regularity results for fully nonlinear integro-differential operators
with\\ nonsymmetric positive kernels }
\author{ Yong-Cheol Kim and Ki-Ahm Lee }
\begin{document}
\begin{abstract}
 In this paper, we consider fully nonlinear integro-differential
 equations with possibly nonsymmetric kernels. We are able to find
 different versions of Alexandroﬀ-Backelman-Pucci estimate
 corresponding to the full class $\cS^{\fL_0}$ of uniformly elliptic  nonlinear equations with
 $1<\sigma<2$ (subcritical case) and to their subclass $\cS^{\fL_0}_{\eta}$ with $0<\sigma\leq  1$.
 We show that $\cS^{\fL_0}_{\eta}$ still includes a large number of nonlinear operators as well as linear operators.
 And we show a Harnack inequality, H\"older regularity, and
 $C^{1,\alpha}$-regularity of the solutions by obtaining decay
 estimates of their level sets in each cases.
\end{abstract}
\thanks {2000 Mathematics Subject Classification: 47G20, 45K05,
35J60, 35B65, 35D10 (60J75)}
\address{$\bullet$ Yong-Cheol Kim : Department of Mathematics Education, Korea University, Seoul 136-701,
Korea }

\email{ychkim@korea.ac.kr}

\address{$\bullet$ Ki-Ahm Lee : Department of Mathematics, Seoul National University, Seoul 151-747,
Korea} \email{kiahm@math.snu.ac.kr}

\maketitle

\section{Introduction}\label{sec-1}

In this paper, we are going to consider the regularity of the
viscosity solutions of {\it integro-differential operators} with
possibly nonsymmetric kernel:
\begin{equation}\label{eq-2.1}
\cL u(x)=\pv \int_{\BR^n}\mu(u,x,y)K(y)\,dy
\end{equation}
where $\mu(u,x,y)=u(x+y)-u(x)-(\n u(x)\cdot y)\chi_{B_1}(y)$, which
describes  the infinitesimal generator of given purely jump
processes, i.e. processes without diffusion or drift part \cite{CS}.
We refer the detailed definitions of notations to \cite{KL}. Then we
see that $\cL u(x)$ is well-defined provided that
$u\in\rC^{1,1}(x)\cap\rB(\BR^n)$ where $\rB(\BR^n)$ denotes {\it the
family of all real-valued bounded functions defined on $\BR^n$}. If
$K$ is symmetric (i.e. $K(-y)=K(y)$), then an odd function
$\bigl[(\n u(x)\cdot y)\chi_{B_1}(y)\bigr]K(y)$ will be canceled in
the integral, and so we have that
$$\cL  u(x)=\pv\int_{\BR^n}\bigl[u(x+y)+u(x-y)-2u(x)\bigr]K(y)\,dy .$$
On the other hand, if $K$ is not symmetric, the effect of $\bigl[(\n
u(x)\cdot y)\chi_{B_1}(y)\bigr]K(y)$ persists and we can actually
observe that the influence of this gradient term becomes stronger as
we try to get an estimate in smaller regions.

Nonlinear integro-differential operators come from the  stochastic
control theory related with
\begin{equation*}
\cI u(x)=\sup_{\ap}\cL_{\ap}u(x),
\end{equation*}
or game theory associated with
\begin{equation}\label{eq-1.4}
\cI u(x)=\inf_{\bt}\sup_{\ap}\cL_{\ap\bt}u(x),
\end{equation}
when the stochastic process is of L\`evy type allowing jumps; see
\cite{S, CS, KL}. Also an operator like $\cI
u(x)=\sup_{\ap}\inf_{\bt}\cL_{\ap\bt}u(x)$ can be considered.
Characteristic properties of these operators can easily be derived
as follows;
\begin{equation}\label{eq-1.5}
\begin{split}\inf_{\ap\bt}\cL_{\ap\bt}
v(x)&\le\cI [u+v](x)-\cI u(x)\le\sup_{\ap\bt}\cL_{\ap\bt} v(x).
\end{split}
\end{equation}

\subsection{operators}
In this section, we introduce a class of operators. All notations and
the concepts of viscosity solution follows \cite{KL} where a more
general class of operators has been considered.  Similar concepts can
be found at \cite{CS} for
symmetric kernel.

For our purpose, we shall restrict our attention to the operators
$\cL$ where the measure $\fm$ is given by a positive kernel $K$
which is not necessarily symmetric. That is to say, the operators
$\cL$ are given by
\begin{equation}\label{eq-2.1}
\cL u(x)=\pv \int_{\BR^n}\mu(u,x,y)K(y)\,dy
\end{equation}
where $\mu(u,x,y)=u(x+y)-u(x)-(\n u(x)\cdot
y)\chi_{B_1}(y)$.

And we consider the class $\fL_0$ of the operators $\cL$ associated
with the measures $\fm$ given by positive kernels $K\in\cK_0$
satisfying that
\begin{equation}\label{eq-2.3}(2-\sm)\f{\ld}{|y|^{n+\sm}}\le
K(y)\le(2-\sm)\f{\Ld}{|y|^{n+\sm}},\,\,0<\sm<2.
\end{equation}
The maximal operator and the minimal operator with respect to a
class $\fL$ of linear integro-differential operators are defined by
\begin{equation}\label{eq-3.1}\cM^+_{\fL}u(x)=\sup_{\cL\in\fL}\cL u(x)\,\,\text{ and }\,\,\cM^-_{\fL}u(x)=\inf_{\cL\in\fL}\cL u(x).
\end{equation}

We say that $P$ is a {\it paraboloid of opening $M$} if
$P(x)=\el_0+\el(x)\pm\f{M}{2}\,|x|^2$ where $M$ is a positive
constant, $\el_0$ is real constant and $\el$ is a linear function.
Then $P$ is called {\it convex} when we have $+$ in the above and
{\it concave} when we have $-$ in the above. Let $\Om\subset\BR^n$
be a bounded domain. Given two semicontinuous functions $u,v$
defined on an open subset $U\subset\Om$ and a point $x_0\in U$, we
say that {\it $v$ touches $u$ by above at $x_0\in U$} if
$u(x_0)=v(x_0)$ and $u(x)\le v(x)$ for any $x\in U$. Similarly, we
say that {\it $v$ touches $u$ by below at $x_0\in U$} if
$u(x_0)=v(x_0)$ and $u(x)\ge v(x)$ for any $x\in U$. For a
semicontinuous function $u$ on $\Om$ and an open subset $U$ of
$\Om$, we define $\Theta^+ (u,U)(x_0)$ to be the infimum of all
positive constants $M$ for which there is a convex paraboloid of
opening $M$ that touches $u$ by above at $x_0\in U$. Also we define
$\Theta^+ (u,U)(x_0)=\iy$ if no such constant $M$ exists. Similarly,
we define $\Theta^- (u,U)(x_0)$ to be the infimum of all positive
constants $M$ for which there is a concave paraboloid of opening $M$
that touches $u$ by below at $x_0\in U$, and also we define
$\Theta^- (u,U)(x_0)=\iy$ if no such constant $M$ exists. Finally we
set $\Theta(u,U)(x_0)=\max\{\Theta^+ (u,U)(x_0),\Theta^-
(u,U)(x_0)\}\le\iy.$ For these definitions, the readers can refer to
\cite{CC}.
\begin{definition}\label{def-2.1}
Let $\Om\subset\BR^n$ be a bounded domain and let $u:\Om\to\BR$ be a
semicontinuous function. Given $x_0\in\Om$, we say that $u$ is
$\rC_{\pm}^{1,1}$ {\rm at $x_0$} $($ resp. $\rC^{1,1}$ {\rm at
$x_0$} $)$ if $\,\Theta^{\pm} (u,U)(x_0)<\iy$ $($ resp.
$\Theta(u,U)(x_0)<\iy$ $)$ for an open neighborhood $U$ of $x_0$ and
we write $u\in\rC^{1,1}[x_0]$ if $\,\Theta(u,U)(x_0)<\iy$. Given a
fixed $\ep\in (0,1)$, we set $\Theta(u,\ep)(x)=\Theta(u,\Om\cap
B_{\ep}(x))$ for $x\in\Om$ and we write $u\in\rC^{1,1}[\Om]$ if
$\sup_{x\in\Om}\Theta(u,\ep)(x)\fd\Theta_{\Om}[u]<\iy$.
\end{definition}
\begin{remark}\label{rem-2.2}
If $u\in\rC^{1,1}[x_0]$ then $u$ is differentiable at $x_0$ because
$u$ lies between two tangent paraboloids in an open neighborhood of
$x_0$.
\end{remark}
\begin{definition}\label{def-2.3}
A function $u:\BR^n\to\BR$ is said to be $\text{\rm $\rC^{1,1}$ at
$x\in\BR^n$}$ $($we write $u\in\rC^{1,1}(x)$$)$, if there are a
vector $v\in\BR^n$, $r_0>0$ and $M>0$ such that
\begin{equation}\label{eq-2.4}
\bigl|u(x+y)-u(x)-v\cdot y\bigr|\le M\,|y|^2\,\,\,\text{ for any
$y\in B_{r_0}$.}
\end{equation}
We write $u\in\rC^{1,1}(U)$ if $\,u\in\rC^{1,1}(x)$ for any $x\in U$
and the constant $M$ in \eqref{eq-2.4} is independent of $\,x$,
where $U$ is an open subset of $\BR^n$.
\end{definition}
\begin{remark}\label{rem-2.5}
\begin{enumerate}[$(a)$]
\item  Such vector $v$ exists uniquely and moreover $v=\n u(x)$.
\item If $\,u\in\rC^{1,1}[x]$ for $x\in\Om$, then we see that
$u\in\rC^{1,1}(x)$. Moreover, it is easy to show that the converse
holds. Thus we have $\rC^{1,1}[\Om]=\rC^{1,1}(\Om)$.
\end{enumerate}
\end{remark}

For $x\in\Om$ and a function $u:\BR^n\to\BR$ which is semicontinuous
on $\overline\Om$, we say that $\vp$ belongs to the function class
$\rC^2_{\Om}(u;x)^+$ (resp. $\rC^2_{\Om}(u;x)^-$) and we write
$\vp\in\rC^2_{\Om}(u;x)^+$ (resp. $\vp\in\rC^2_{\Om}(u;x)^-$) if
there are an open neighborhood $U\subset\Om$ of $x$ and
$\vp\in\rC^2(U)$ such that $\vp(x)=u(x)$ and $\vp>u$ (resp. $\vp<u$)
on $U\s\{x\}$. We note that geometrically $u-\vp$ having a local
maximum at $x$ in $\Om$ is equivalent to $\vp\in\rC^2_{\Om}(u;x)^+$
and $u-\vp$ having a local minimum at $x$ in $\Om$ is equivalent to
$\vp\in\rC^2_{\Om}(u;x)^-$. For $x\in\Om$ and $\vp\in
\rC^2_{\Om}(u;x)^{\pm}$, we write
$$\mu(u,x,y;\n\vp)=u(x+y)-u(x)-(\n\vp(x)\cdot y)\chi_{B_1}(y),$$
and the expression for $\cL_{\ap\bt}\,u(x;\n\vp)$ and $\cI
u(x;\n\vp)$ may be written as
\begin{equation*}\begin{split}\cL_{\ap\bt}\,u(x;\n\vp)&=\int_{\BR^n}\mu(u,x,y;\n\vp)K_{\ap\bt}(y)\,dy,\\
\cI
u(x;\n\vp)&=\inf_{\bt}\sup_{\ap}\cL_{\ap\bt}\,u(x;\n\vp),\end{split}\end{equation*}
where $K_{\ap\bt}\in\cK_0$. Then we see that
$\cM^-_{\fL_0}u(x;\n\vp)\le\cI
u(x;\n\vp)\le\cM^+_{\fL_0}u(x;\n\vp)$, and $\cM^+_{\fL_0}u(x;\n\vp)$
and $\cM^-_{\fL_0}u(x;\n\vp)$ have the following simple forms;
\begin{equation}\label{eq-3.2}
\begin{split}
&\cM^+_{\fL_0}u(x;\n\vp)=(2-\sm)\int_{\BR^n}\f{\Ld\mu^+(u,x,y;\n\vp)
-\ld\mu^-(u,x,y;\n\vp)}{|y|^{n+\sm}}\,dy,\\
&\cM^-_{\fL_0}u(x;\n\vp)=(2-\sm)\int_{\BR^n}\f{\ld\mu^+(u,x,y;\n\vp)
-\Ld\mu^-(u,x,y;\n\vp)}{|y|^{n+\sm}}\,dy,
\end{split}
\end{equation}
where $\mu^+$ and $\mu^-$ are given by
\begin{equation*}\begin{split}\mu^{\pm}(u,x,y;\n\vp)&=\max\{\pm
    \mu(u,x,y;\n\vp),0\}.
\end{split}\end{equation*} We note
if $u\in\rC^{1,1}[x]$, then $\cI u(x;\n\vp)=\cI u(x)$ and
$\cM^{\pm}_{\fL_0}u(x;\n\vp)=\cM^{\pm}_{\fL_0}u(x)$. We shall use
these maximal and minimal operators to obtain regularity estimates.

Let $K(x)=\sup_{\ap}K_{\ap}(x)$ where $K_{\ap}$'s are all the
kernels of all operators in a class $\fL$. For any class $\fL$, we
shall assume that
\begin{equation}\label{eq-3.3}
\int_{\BR^n}(|y|^2\wedge 1)\,K(y)\,dy<\iy.
\end{equation}
 Using
the extremal operators, we provide a general definition of
ellipticity for nonlocal equations. The following is a kind of
operators of which the regularity result shall be obtained in this
paper.

\begin{definition}\label{def-3.1} Let $\fL$ be a class of linear
integro-differential operators. Assume that \eqref{eq-3.3} holds for $\fL$.
Then we say that an operator $\cJ$ is $\text{\rm elliptic with
respect to $\fL$}$, if it satisfies the following properties:

$(a)$ $\cJ u(x)$ is well-defined for any $u\in\rC^{1,1}[x]\cap
\rB(\BR^n)$.

$(b)$ If $\,u\in\rC^{1,1}[\Om]\cap\rB(\BR^n)$ for an open
$\Om\subset\BR^n$, then $\cJ u$ is continuous on $\Om$.

$(c)$ If $\,u,v\in\rC^{1,1}[x]\cap\rB(\BR^n)$, then we have that
\begin{equation}\label{eq-3.4}
\cM^-_{\fL}[u-v](x)\le\cJ u(x)-\cJ
v(x)\le\cM^+_{\fL}[u-v](x).
\end{equation}
And We denote by $\cS^{\fL}$ the class of integro-differential
operators which is elliptic with respect to $\fL$.
\end{definition}

The concept of viscosity solutions, its comparison principle and
stability properties can be found in \cite{CS} for symmetric kernels
and in \cite{KL} for possibly nonsymmetric kernels. Kim and Lee
\cite{KL} considered much larger class of operators but prove the
regularity of viscosity solutions only for $1<\sigma<2$.

Now we are going to consider a subclass $\cS^{\fL}_{\eta}$ of
$\cS^{\fL}$ where the drift effect created by the nonsymmetric
kernel is controllable. For $x\in B_R$ and
$\vp\in\rC^2_{B_R}(u;x)^{\pm}$, we set
$$\mu_R(u,x,y;\n\vp) =u(x+y)-u(x)-(\n\vp(x)\cdot y)\,\chi_{B_R}(y).$$
For $u\in\rC^{1,1}[x]$, we write $\mu_R(u,x,y)=\mu_R(u,x,y;\n u)$.
Then we define $\mu^{\pm}_R$ and $\cM^{\pm}_{\fL_0,R} u(x;\n\vp)$ by
replacing $\mu$ by $\mu_R$ in the definition $\cM^{\pm}_{\fL_0}
u(x;\n\vp)$. We note if $u\in\rC^{1,1}[x]$, then
$\cM^{\pm}_{\fL_0,R}u(x;\n\vp)=\cM^{\pm}_{\fL_0,R}u(x)\fd\pv\int_{\BR^n}\mu_R(u,x,y)K(y)\,dy$.

\begin{definition}\label{def-S-eta}
Let $0<\eta\leq 1$ and  $\cI\in\cS^{\fL}$, where $\fL$ is a class of
linear integro-differential operators. Then we say that $\cI\in
\cS^{\fL}_{\eta}$ if, for $R\in(0,1]$, there are $\cB^{\pm}_R:
\R^n\rightarrow \R$ such that
\begin{enumerate}
\item $\cB^{\pm}_R(\cdot)$ is homogeneous of degree one, i.e. $\cB^{\pm}_R(0)=0$ and $\cB^{\pm}_R(\n u)=\cB^+_R\left( \f{\n u}{|\n u|}\right)|\n u|$
for $|\n u|\neq 0$,
\item $|\{\nu\in S^{n-1}:\, \cB_R^{\pm}(\nu)<0\}|\geq \eta|S^{n-1}|>0$,
\item $\cM_{\fL,R,\e}^- u(x)\le \cI u(x)-\cI 0 (x)\le \cM^+_{\fL,R,\e}u(x)$ whenever $u\in\rC^{1,1}[x]\cap\rB(\BR^n)$ for $x\in
B_R$,
\end{enumerate}
\item
where
$\cM_{\fL,R,\eta}^{\pm}u(x):=\cM_{\fL,R}^{\pm}u(x)\pm\cB^{\pm}_R\left(
\n u(x)\right)\pm(2-\sigma)CR^{1-\sigma}|\n u(x)|.$
\begin{definition} Let $\cL\in\fL$ be a linear integro-differential operator with a kernel $K$. For $0<R<1$, the drift vector $\fb_{\cL,R}   $
of $\cL$ at $R$ is defined by
$$\fb_{\cL,R}= (2-\sm)\int_{B_1\s B_R}y\,K(y)\,dy.$$
\end{definition}
\end{definition}

\begin{lemma}\label{lem-subclass-1}
Let $0<\sigma<2$ and $0<R<1$. Let $\fL$ be a class of linear
integro-differential operators. Then we have the following results:
\begin{enumerate}
\item If $\cL$ is a linear integro-differential operator which is in $\fL$, then $\cL\in\cS^{\fL}_{\eta}$ for $0<\eta\leq\frac12.$
\item  Let $\cL_i$ be  a linear  integro-differential operator with a kernel $K_i$  for $i=1,\cdots,N$.
Let $\fb_{\cL_i,R}$ be the drift vectors of $\cL_i$.
        Assume that there is a unit vector $\fa$ such that
       for any nonzero drift vector $\fb_{\cL_i,R}$, $$\left\langle \fa,\frac{\fb_{\cL_i,R}}{| \fb_{\cL_i,R}|}\right\rangle >0.$$
       If $\cI\in\cS^{\fL}$ satisfies that $$ \min _{i=1,\cdots,N}\cL_i u(x) \leq \cI u(x)-\cI 0(x)\leq\max _{i=1,\cdots,N}\cL_i
         u(x)$$ whenever $u\in\rC^{1,1}[x]\cap\rB(\BR^n)$ for $x\in
         B_R$, then $\cI\in \cS^{\fL}_{\eta}$ for some $\eta>0$.
\item  Let $\cL_{\alpha}$ be a linear integro-differential operator with a kernel $K_{\alpha}$, $\alpha\in I$.
Let $\fb_{\cL_{\alpha},R}$ be the drift vectors of $\cL_{\alpha}$.
        Assume that there is a vector $\fa\in S^{n-1}$ and $\eta>0$ such that
       for any nonzero drift vectors $\fb_{\cL_{\alpha},R}$ , $$\left\langle \fa,\frac{\fb_{\cL_{\alpha},R}}{| \fb_{\cL_{\alpha},R}|}\right\rangle
       \geq \eta^{\frac{1}{n-1}}\quad\text{for any $\alpha\in I$}.$$
       If $\cI\in\cS^{\fL}$ satisfies that $$ \min _{\alpha\in I}\cL_{\alpha}u(x)  \leq \cI u(x)-\cI 0(x)  \leq  \max _{\alpha\in
         I}\cL_{\alpha}u(x)$$ whenever $u\in\rC^{1,1}[x]\cap\rB(\BR^n)$ for $x\in
         B_R$, then we have $\cI\in \cS^{\fL}_{\eta}$  .
\end{enumerate}
\end{lemma}
\begin{proof} (1) Take $u\in\rC^{1,1}[x]\cap\rB(\BR^n)$ for $x\in
B_R$. Then we have that
\begin{equation*}
\begin{split}
\cL u(x)&=(2-\sm)\int_{\BR^n}\mu(u,x,y)K(y)\,dy.\\
&=(2-\sm)\int_{\BR^n}\mu_R(u,x,y)K(y)\,dy+(2-\sm)\int_{B_1\backslash B_R}(y\cdot\n u(x) )K(y)\,dy.\\
&\leq  \cM^+_{\fL,R}u(x)+\cB^+_R\left(\n u(x)\right)
\end{split}
\end{equation*}
for $\cB^+_R\left(\fa\right)=\fb_{\cL,R}\cdot\fa$. And  if
$\fb_{\cL,R}\neq 0$, $|\{\fa\in S^{n-1}:\cB^+_R\left(\fa\right)\leq
0\}|=\frac12 |S^{n-1}|$,\, which means $0<\eta\leq\frac12$;
otherwise, $\eta=1$.

(2) Since there is a finite number of unit vectors
$\frac{\fb_{\cL_i,R}}{| \fb_{\cL_i,R}|}$, we can find $\eta>0$ such
that
$$\left\langle \nu,\frac{\fb_{\cL_{i},R}}{| \fb_{\cL_{i},R}|}\right\rangle \geq \eta^{\frac{1}{n-1}}\quad\text{ for $i=1,\cdots,n$ and $\fb_{\cL_{i},R}\neq 0$ }.$$
Take $u\in\rC^{1,1}[x]\cap\rB(\BR^n)$ for $x\in B_R$. Then we have
that
\begin{equation*}
\begin{split}
\max _{i=1,\cdots,N}\cL_i u(x)&=(2-\sm)\max _{i=1,\cdots,N}\int_{\BR^n}\mu(u,x,y)K_i(y)\,dy.\\
&=(2-\sm)\max _{i=1,\cdots,N}\int_{\BR^n}\mu_R(u,x,y)K_i(y)\,dy\\
 &\quad+(2-\sm)\max _{i=1,\cdots,N}\int_{B_1\backslash B_R}(y\cdot\n u(x)) K_i(y)\,dy.\\
&\leq  \cM^+_{\fL,R}u(x)+\cB^+_R\left(\n u(x)\right)
\end{split}
\end{equation*}
for $\cB^+_R\left(\fa\right)=\max
_{i=1,\cdots,N}(\fb_{\cL_i,R}\cdot\fa)$. Also we have that
$|\{\fa:\cB^+_R\left(\fa\right)\leq 0\}|\geq\eta |S^{n-1}|$.

(3) It can be obtained from the same argument as (2).
\end{proof}
\begin{lemma}\label{lem-subclass-2}
\item
\begin{enumerate} \item Let $0<\sm<2$ and let $\fL$ be a class of linear
integro-differential operators. If $\cL\in \cS^{\fL}$ with a
symmetric kernel $K$, then $\cL\in\cS^{\fL}_{\eta}$ for some
$\e\in(0,1]$.
          In addition, if $\cL_{\alpha}$ has symmetric kernel for all $\alpha$,
          then $\sup_{\alpha}\cL_{\alpha}, \inf_{\alpha}\cL_{\alpha}\in S^{\fL}_{\eta}$ for $1\geq\eta>0$.
\item If $1<\sigma<2$, then $\cS^{\fL_0}=\cS^{\fL_0}_{\eta}$ for some $\e\in(0,1]$.
\end{enumerate}
\end{lemma}
\begin{proof}
(1) If $K(y)$ is symmetric i.e. $K(-y)=K(y)$, then $\fb_{\cL,R}=0$.
If $u\in\rC^{1,1}[x]\cap\rB(\BR^n)$ for $x\in B_R$, then
$\int_{\BR^n}\mu(u,x,y)K(y)\,dy=\int_{\BR^n}\mu_R(u,x,y;\n
u)K(y)\,dy$ for all $R>0$, which implies the conclusion with
$\cB^\pm_R\left(\n u\right)=0$.

(2) By the definition, $\cS^{\fL_0}_{\eta}\subset\cS^{\fL_0}$.  Take
any $\cI\in\cS^{\fL_0}$ and  $u\in\rC^{1,1}[x]\cap\rB(\BR^n)$ for
$x\in B_R$. We set
$$b_R(x)=(2-\sigma)\int_{B_1\backslash B_R}\f{y(\Ld\chi_{\mu>0}+\ld\chi_{\mu\leq 0})}{|y|^{n+\sm}}\,dy.$$
Then we easily obtain that, for $1<\sigma<2$,
$$|b_R(x)|\leq (2-\sigma)C R^{1-\sigma}.$$
Therefore we have that
\begin{equation*}\begin{split}
\cI u(x)-\cI 0(x)&\le \cM^+_{\fL_0} u(x)=\cM^+_{\fL_0,R}
u(x)+(2-\sm)\int_{B_1\backslash B_R}\f{\Ld(\mu^+-\mu_R^+)
-\ld(\mu^- - \mu_R^-)}{|y|^{n+\sm}}\,dy\\
&\leq \cM^+_{\fL_0,R} u(x)-b_R(x)\cdot \n u(x)\\
&\leq \cM^+_{\fL_0,R} u(x)+(2-\sigma)C R^{1-\sigma}|\n u(x)|.\\
\end{split}\end{equation*}
The lower bound can be obtained similarly. Therefore
$\cI\in\cS^{\fL_0}_{\eta}$ for $0<\eta\leq 1$.
\end{proof}
\subsection{Main equation}
The natural Dirichlet problem for such a nonlocal
operator $\cI$. Let $\Om$ be an open domain in $\BR^n$. Given a
function $g$ defined on $\BR^n\s\Om$, we want to find a function
$u$ such that
$$\begin{cases}\cI u(x)=0 &\text{ for any $x\in\Om$,}\\u(x)=g(x) &\text{
for $x\in\BR^n\s\Om$.}\end{cases}$$ Note that the boundary condition
is given not only on $\pa\Om$ but also on the whole complement of
$\Om$. This is because of the nonlocal character of the operator
$\cI$. From the stochastic point of view, it corresponds to the
fact that a discontinuous L\`evy process can exit the domain $\Om$
for the first time jumping to any point in $\BR^n\s\Om$.

In this paper, we shall concentrate mainly upon the regularity
properties of viscosity solutions to an equation $\cI u(x)=0$. We
shall briefly give a very general comparison principle from which
existence of the solutions can be obtained in smooth domains. Since
kernels of integro-differential operators are comparable to the
kernel of the fractional Laplace operator $-(-\Delta)^{\sm/2}$, the
theory we want to develop can be understood as a theory of viscosity
solutions for fully nonlinear operators of fractional order.

The differences between local and nonlocal operators have been
discussed at \cite{KL}.
\subsection{Known results and Key Observations}

There are some known results about Harnack inequalities and H\"older
estimates for integro-differential operators with positive symmetric
kernels (see \cite{J} for analytical proofs and \cite{BBC},
\cite{BK1}, \cite{BK2},\cite{BL}, \cite{KS}, \cite{SV} for
probabilistic proofs).
 The estimates in all these previous results
blow up as the index $\sm$ of the operator approaches $2$. In this
respect, they do not generalize to elliptic partial differential
equations. However there is some known result on regularity results
for fully nonlinear integro-differential equations associated with
nonlinear integro-differential operators with positive symmetric
kernels which remain uniform as the index $\sm$ of the operator
approaches $2$ (see \cite{CS}). Therefore these results make the
theory of integro-differential operators and elliptic differential
operators become somewhat unified. For nonlinear
integro-differential operators with possibly nonsymmetric kernels,
the authors introduced larger class of operators and proved Harnack
inequalities and H\"older estimates when $1<\sigma<2$ (see
\cite{KL}).

In this paper, we  are going to consider  nonlinear
integro-differential operators with possibly nonsymmetric kernels,
when $0<\sigma<2$.

Throughout this paper we would like to briefly present the
necessary definitions and then prove some regularity estimates.
Our results in this paper are

$\bullet$ A nonlocal version of the Alexandroff-Backelman-Pucci
estimate for fully nonlinear integro-differential equations.

$\bullet$ A Harnack inequality, H\"older regularity and an interior
$\rC^{1,\ap}$-regularity result for certain fully nonlinear
integro-differential equations.

Key observations are the following:

$\bullet$  For the nonsymmetric case, $K(y)$ and $K(-y)$ can be
chosen any of $\ld/|y|^{n+\sm}$ or
  $\Ld/|y|^{n+\sm}$. Therefore there could be an extra
  term $\ds\int_{\BR^n}\frac{\left|(\n u(x)\cdot y)\chi_{B_1}(y)\right|}{|y|^{n+\sigma}}dy$.

$\bullet$ The equation  is not scaling invariant due to
$|\chi_{B_1}(y)|$.

$\bullet$ Somehow the equation has a drift term, not only the diffusion term.
                The case $1<\sm<2$ and the case $0<\sm\le 1$ require different technique
                 due to the difference of the blow rate as $|y|$ approaches to zero and
                 the decay rate as $|y|$ approaches to infinity.
 When $1<\sm<2$, a controllable decay rate of kernel allows H\"older
 regularities in a larger class, which is invariant under an one-sided
                scaling i.e. if $u$ is a solution of the homogeneous
                equation, then so is  $u_{\ep}(x)=\ep^{-\sm}u(\ep
                x)$ for $0<\ep\le 1$.
                Critical case ($\sm=1$) and supercritical case
                ($0<\sm<1$) have been studied in \cite{BBC} with different
                techniques due to the slow decay rate of the kernel as
                $|x|\rightarrow \infty$.

\subsection{Outline of Paper}
In Section \ref{sec-6},  we show various  nonlocal versions of the
Alexandroff-Backelman-Pucci estimate  to handle the difficulties
caused by the gradient effect. It has different orders at
subcritical, critical  and supercritical cases.  In Section
\ref{sec-7}, we construct a special function and apply A-B-P
estimates to obtain the decay estimates of upper level sets which is
essential in proving H\"older estimates in Section \ref{sec-9}.

In
Section \ref{sec-8.0}, we prove a Harnack inequality which plays an
important role in analysis. And then the H\"older estimates and  an interior
$\rC^{1,\ap}$-estimates come from the arguments at \cite{CS,KL}.

\section{ A nonlocal Alexandroff-Bakelman-Pucci estimate}\label{sec-6}

The Alexandroff-Bakelman-Pucci (A-B-P) estimate plays an important
role in  Krylov and Sofonov theory \cite{KS} on Harnack inequality
for linear uniformly elliptic equations with measurable
coefficients. The concept of viscosity solution is given pointwise
through touching test function; see \cite{KL}. A-B-P estimate tells
us that the maximum value is controlled by an integral quantity of
the source term on the contact set, which will give us key lemma
(Lemma \ref{lem-6.1}) saying that the pointwise value of nonnegative
function gives the lower bound of the measure of lower level set. We
employ measure theoretical version of A-B-P estimate introduced at
\cite{CS} and extended to nonsymmetric case at \cite{KL}.

New A-B-P estimates below are two main differences from the
arguments at \cite{CS,KL}.
\begin{itemize}
\item The operators considered at \cite{CS,KL} are scaling invariant,
  but \eqref{eq-2.1} doesn't have such property due to $\chi_{B_1}(y)$ in the gradient term.
  So we keep the size of the domain $B_R$ at the following estimates.
\item The control of bad set, Lemma \ref{lem-6.1}, deteriorates as $R\rightarrow 0$ since
  $R^{\sigma -2}J_{\sigma}(R)$ goes to $\infty$ as $R\to 0$. A-B-P estimate will be used to prove
  key Lemma \ref{lem-7.4} where we have an extra term $R$ to subdue the
  blow-up rate. But we have still $R\times R^{\sigma -2}J_{\sigma}(R)
  \approx R^{\sigma -1}$ (for $0<\sigma<1$) and $-log(R)$ (for $\sigma=1$)
  which blows up when $0<\sigma\leq 1$. So we introduced a subclass $\cS_{\eta}$ of $\cS$ and a different
  version of A-B-P estimate (Lemma \ref{lem-6.7}) for $0<\sigma\leq  1$ where we have better
  control of gradient term. Still $\cS_{\eta}$ includes a large class of nonlinear operators, Lemma \ref{lem-subclass-1},\ref{lem-subclass-2}.
\end{itemize}

Let $R\in(0,R_0]$ for some $R_0\in(0,1)$ (in fact, the existence of
$R_0$ was mentioned in \cite{KL}) and let $u:\BR^n\to\BR$ be a
function which is not positive outside the ball $B_{R/2}$ and is
upper semicontinuous on $\overline B_R$. We consider its {\it
concave envelope} $\Gm$ in $B_{2R}$ defined as
$$\Gm(x)=\begin{cases}\inf\{p(x): p\in\Pi,\,p>u^+\,\,\text{in $B_{2R}$}\}
&\text{ in $B_{2R}$,}\\0 &\text{ in $\BR^n\s B_{2R}$,}\end{cases}$$
where $\Pi$ is the family of all the hyperplanes in $\BR^n$. Also we
denote the {\it contact set} of $u$ and $\Gm$ in $B_R$ by
$\cC(u,\Gm,B_R)=\{y\in B_R:u(y)=\Gm(y)\}$.

\subsection{A-B-P estimate with blow-up rate}\label{sec-abp-blow}

\begin{lemma}\label{lem-6.1} Let $0<\sigma<2$ and $0<R\le R_0$.
Let $u\le 0$ in $\BR^n\s B_R$ and let $\Gm$ be its concave envelope
in $B_{2R}$. If $\,u\in \rB(\BR^n)$ is a viscosity subsolution to
$\cM^+_{\fL_0} u=-f$ on $B_R$ where $f:\BR^n\to\BR$ is a function
with $f>0$ on $\cC(u,\Gm,B_R)$, then there exists some constant
$C>0$ depending only on $n,\ld$ and $\Ld$ $($but not on $\sm$$)$
such that for any $x\in\cC(u,\Gm,B_R)$ and any $M>0$ there is some
$k\in\BN\cup\{0\}$ such that

\begin{equation}\label{eq-6.1}
\bigl|\{y\in R_k: \mu^-(u,x,y;\n\Gm)\geq  M_0 r_k^2\}\bigr|\le
C\f{R^{\sigma-2}(f(x)+J_{\sigma}(R)|\n\Gm(x)|)}{M}|R_k|
\end{equation} where $R_k=B_{r_k}\s B_{r_{k+1}}$ for $r_k=\vr_0
2^{-\f{1}{2-\sm}-k}R$, $\vr_0=1/(16\sqrt n)$ and $J_{\sigma}(R)$ is
$\frac{1}{1-\sigma}(1-R^{1-\sigma})$ for $\sm\in(0,1)\cup(1,2)$ and
$-\log (R)$ for $\sigma=1$. Here, $\n\Gm(x)$ denotes any element of
the superdifferential $\pa\Gm(x)$ of $\,\Gm$ at $x$.
\end{lemma}

\begin{proof}  Let $0<\sigma<2$ and
$0<R\le R_0$. Take any $x\in\cC(u,\Gm,B_R)$. Since $u$ can be
touched by a hyperplane from above at $x$, we see that
$\n\vp(x)=\n\Gm(x)$ for some $\vp\in \rC^2_{B_R}(u;x)^+$. Thus
$\cM^+_{\fL_0} u(x;\n\Gm)$ is well-defined and we have that
$$\cM^+_{\fL_0} u(x;\n\Gm)=(2-\sm)\int_{\BR^n}\f{\Ld\mu^+(u,x,y;\n\Gm)
-\ld\mu^-(u,x,y;\n\Gm)}{|y|^{n+\sm}}\,dy.$$ We note that
$\mu(u,x,y;\n\Gm) =u(x+y)-u(x)-(\n\Gm(x)\cdot y )\chi_{B_1}(y)\le 0$
for any $y\in B_R$ by the definition of concave envelope of $u$ in
$B_{2R}$. Since $\mu^+(u,x,y;\n\Gm)\leq |\n \Gm(x)||y|
\chi_{B_1}(y)$ for any $y\in \BR^n \s B_R$, we have that
\begin{equation}
\begin{split}
\int_{\BR^n}\f{\Ld\mu^+(u,x,y;\n\Gm)}{|y|^{n+\sm}}\,dy&\le
\int_{B_1\s
B_R}\f{\Ld |\n\Gm(x)|| y|}{|y|^{n+\sm}}\,dy\\
&=\om_n\Ld\,J_{\sigma}(R) |\n\Gm(x)|
\end{split}
\end{equation}
where $\om_n$ denotes the surface area of $S^{n-1}$ and
\begin{equation}
\begin{split}
J_{\sigma}(R)=
\begin{cases}
 \frac1{1-\sigma}\left(1-R^{1-\sigma}\right)&\text{for $\sm\in(0,1)\cup(1,2),$}\\
-\log (R) &\text{for $\sigma=1$.}
\end{cases}
\end{split}
\end{equation}
Here we see that $|J_{\sm}(R)|$ is finite for $0<\sigma<2$. Thus it
follows from simple calculation that
\begin{equation*}
\begin{split}-f(x)&\le\cM^+_{\fL_0} u(x;\n\Gm)\\&=(2-\sm)\biggl(\,\int_{\BR^n}\f{-\ld\mu^-(u,x,y;\n\Gm)}{|y|^{n+\sm}}\,dy
+\int_{\BR^n}\f{\Ld\mu^+(u,x,y;\n\Gm)}{|y|^{n+\sm}}\,dy\,\biggr)\\
&\le(2-\sm)\int_{B_{r_0}}\f{-\ld\mu^-(u,x,y;\n\Gm)}{|y|^{n+\sm}}\,dy
+(2-\sigma)\om_n\Ld\,J_{\sigma}(R)\,|\n\Gm(x)|
\end{split}
\end{equation*}
for any $x\in\cC(u,\Gm,B_R)$, where $r_0=\vr_0 2^{-\f{1}{2-\sm}}R$.
Splitting the above integral in the rings $R_k$, we have that
\begin{equation}\label{eq-6.2}
 f(x)\ge(2-\sm)\ld\sum_{k=0}^{\iy}\int_{R_k}\f{\mu^-(u,x,y;\n\Gm)}{|y|^{n+\sm}}\,dy
-(2-\sm)\om_n\Ld J_{\sigma}(R)  |\n\Gm(x)|.
\end{equation}

Assume that the conclusion \eqref{eq-6.1} does not hold, i.e. for
any $C>0$ there are some $x_0\in\cC(u,\Gm,B_R)$ and $M_0>0$ such
that
\begin{equation*}\label{eq-6.3}
\bigl|\{y\in R_k: \mu^-(u,x_0,y;\n\Gm)\geq  M_0
r_k^2\}\bigr|>C\f{R^{\sigma-2}\bigl(f(x_0)+ J_{\sm}(R)|\n\Gm
(x_0)|\bigr)}{M_0}|R_k|
\end{equation*}
for all $k\in\BN\cup\{0\}$. Since $(2-\sm)\frac{1}{1-2^{-(2-\sm)}}$
remains bounded below for $\sm\in (0,1]\,$, it thus follows from
\eqref{eq-6.2} that
\begin{equation}
\begin{split}
\f{f(x_0)}{2-\sm}&\ge\ld\sum_{k=0}^{\iy}\int_{R_k}\f{\mu^-(u,x_0,y;\n\Gm(x))}{|y|^{n+\sm}}\,dy-\om_n\Ld J_{\sigma}(R) |\n\Gm(x_0)|\\
&\ge
c\sum_{k=0}^{\iy}M_0\,\f{r_k^2}{r_k^{\sm}}\,CR^{\sigma-2}\,\f{f(x_0)
+J_{\sigma}(R)|\n\Gm(x_0)|}{M_0}-\om_n\Ld J_{\sigma}(R)
|\n\Gm(x_0)|.
\end{split}\end{equation}
Thus this implies that
\begin{equation*}\begin{split}f(x_0)+ (2-\sigma)\om_n\Ld
J_{\sigma}(R) |\n\Gm(x_0)|&\ge \f{c\rho_0^2}{1-2^{-(2-\sm)}} C(
f(x_0)+J_{\sigma}(R) |\n\Gm(x_0)|)\\ &\ge C(f(x_0)+ (2-\sigma)
J_{\sm}(R) |\n\Gm(x_0)|)\end{split}\end{equation*} for any $C>0$.
Taking $C$ large enough, we obtain a contradiction. Hence we are
done.
\end{proof}

\noindent{\it Remark.} Lemma \ref{lem-6.1} would hold for any particular
choice of $\vr_0$ by modifying $C$ accordingly. The particular
choice $\vr_0=1/(16\sqrt n)$ is convenient for the proofs in Section
\ref{sec-7}.

\begin{lemma}\cite{CS}\label{lem-6.2} Let $\Gm$ be a concave function on $B_r(x)$
where $x\in\BR^n$ and let $h>0$. If $|\{y\in
S_r(x):\Gm(y)<\Gm(x)+(y-x)\cdot\n\Gm(x)-h\}|\le\ep\,|S_r(x)|$ for
any small $\ep>0$ where $S_r(x)=B_r(x)\s B_{r/2}(x)$, then we have
$\Gm(y)\ge\Gm(x)+(y-x)\cdot\n\Gm(x)-h$ for any $y\in
B_{r/2}(x)$.\end{lemma}

\begin{cor}\label{cor-6.3}
For any $\ep>0$, there is a constant $C>0$ such that for any
function $u$ with the same hypothesis as Lemma \ref{lem-6.1}, there
is some $r\in (0,\vr_0 2^{-\f{1}{2-\sm}}R)$ such that
\begin{equation*}\begin{split} &\f{|\{y\in S_r(x):u(y)<u(x)+(y-x)\cdot\n\Gm(x)-C\,R^{\sigma-2}(f(x)+J_{\sigma}(R)|\n\Gm(x)|)
r^2\}|}{|S_r(x)|}\le\ep,\\
&\int_{\overline Q}g_{\xi}(\n\Gm (y))\det [D^2 \Gm(y)]^-\,dy\le C
R^{n(\sigma-2)}\sup_{y\in\overline Q}\bigl(J_{\sigma}(R)^n+\xi^{-n}
|f(y)|^n\bigr)\, |Q|
\end{split}\end{equation*}
for any $\e>0$ and any cube $Q\subset B_{r/4}(x)$ with diameter $d$
such that $x\in\overline{Q}$ and $r/4<d<r/2$, where
$\vr_0=1/(16\sqrt n)$ and $g_{\xi}(z)=(|z|^{n/(n-1)}+\xi
^{n/(n-1)})^{1-n}$.\end{cor}
\pf The first part can be obtained by
taking $M=C R^{\sigma-2}( f(x)+J_{\sigma}(R)|\n\Gm(x)|)/\ep$ in
Lemma \ref{lem-6.1}. Also the second part follows as a consequence
of Lemma \ref{lem-6.2} and concavity. If $\Gamma$ is $C^2$, then we
have that
\begin{equation*}\begin{split}
\det[D^2 \Gm(x)]^- &\le C (R^{\sigma-2}
f(x)+R^{\sigma-2} J_{\sigma}(R)|\n\Gm(x)|)^n\\
&\le 4^n C \frac{R^{n(\sigma-2)}
J_{\sigma}(R)^n+\xi^{-n}R^{n(\sigma-2)}|f(x)|^n}{g_{\xi}(\n\Gm(x))}.
\end{split}\end{equation*} Then we have that
$\,\int_{\overline {Q^*}}g_{\xi}(\n\Gm(x))\det[D^2\Gm(x)]^-\,dx\le
4^n C \int_{\overline Q}(R^{n(\sigma-2)}
J_{\sigma}(R)^n+\xi^{-n}R^{n(\sigma-2)}|f(x)|^n)\,dx$ if $\Gm$ is
$C^2$, where $Q^*=\cC(u,\Gm,B_R)\cap\overline Q$. It is also true
for the general concave envelope $\Gm$ through  an approximation
with $C^2$-concave functions.

Take any $y\in\cC(u,\Gm,B_R)\cap Q$ where $Q\subset B_{r/4}(x)$ is a
cube with diameter $d$ such that $x\in\overline Q$ and $r/4<d<r/2$.
Similarly to the above, we can obtain that $\int_{\overline
Q}g_{\xi}(\n\Gm(y))\det[D^2 \Gm(y)]^-\,dy\le 4^n C \int_{\overline
Q}( R^{n(\sigma-2)}
J_{\sigma}(R)^n+\xi^{-n}R^{n(\sigma-2)}|f(y)|^n)\,dy$ because
$\det[D^2 \Gm(\cdot)]^-=0$ $\aee$ on $Q\s\cC(u,\Gm,B_R)$ as in
\cite{CC}. Hence this implies the second part. \qed

We obtain a nonlocal version of Alexandroff-Bakelman-Pucci
estimate in the following theorem.

\begin{thm}\label{thm-6.4} Let $u$ and $\Gm$ be functions as in
Lemma \ref{lem-6.1}. Then there exist a finite family
$\{Q_j\}_{j=1}^m$ of open cubes $Q_j$ with diameters $d_j$ such that

$(a)$ Any two cubes $Q_i$ and $Q_j$ do not intersect, $(b)$
$Q^*_j=\cC(u,\Gm,B_R)\cap\overline Q_j\neq\phi$ for $j$,

$(c)$ $\cC(u,\Gm,B_R)\subset\bigcup_{j=1}^m Q^*_j$, $(d)$
$d_j\le\vr_0 2^{-\f{1}{2-\sm}}R$ where $\vr_0=1/(16\sqrt n)$,

$(e)$ $\int_{\overline Q_j}g_{\xi}(\n\Gm (y))\det (D^2
\Gm(y))^-\,dy\le C R^{n(\sigma-2)}(\sup_{\overline Q_j}(
J_{\sigma}(R)^n+\xi^{-n}|f|^n) |Q^*_j|$,

$(f)$ $|\{y\in 4\sqrt n\,Q_j:u(y)\ge\Gm(y)-C
R^{(\sigma-2)}(\sup_{\overline Q_j} (f+J_{\sigma}(R)|\n\Gm|)
d_j^2\}|\ge\xi_0 |Q_j|$,

\noindent where the constants $C>0$ and $\xi_0>0$ depend on $n,\Ld$
and $\ld$ $($ but not on $\sm$$)$.
\end{thm}

\pf In order to obtain such a family, we start by covering $B_R$
with a tiling of cubes of diameter $\vr_0 2^{-\f{1}{2-\sm}}R$. Then
discard all those that do not intersect $\cC(u,\Gm,B_R)$. Whenever
a cube does not satisfy (e) and (f), we split it into $2^n$ cubes
of half diameter and discard those whose closure does not
intersect $\cC(u,\Gm,B_R)$. Now our goal is to prove that
eventually all cubes satisfy (e) and (f) and this process ends
after a finite number of steps.

Assume that the process does not finish in a finite number of steps.
Then we can have an infinite nested sequence of cubes. The
intersection of their closures will be a point $\hat x$. So we may
choose a sequence $\{x_k\}\subset\cC(u,\Gm,B_R)$ with
$\lim_{k\to\iy}x_k=\hat x$. Since $u(x_k)=\Gm(x_k)$ for all
$k\in\BN$, by the upper semicontinuity of $u$ on $\overline B_R$ we
have that $\Gm(\hat x)=\limsup_{k\to\iy}u(x_k)\le u(\hat x)$. Also
we have that $u(\hat x)\le\Gm(\hat x)$ because $u\le\Gm$ on $B_{2R}$
by the definition of the concave envelope $\Gm$ in $B_{2R}$. Thus we
obtain that $u(\hat x)=\Gm(\hat x)$. We will now get a contradiction
by showing that eventually one of these cubes containing $\hat x$
will not split.

Take any $\ep>0$. Then by Corollary \ref{cor-6.3} there is a radius
$r\in (0,\vr_0 2^{-\f{1}{2-\sm}}R)$ such that
\begin{equation*}
\begin{split} &\f{|\{y\in
S_r(\hat x):u(y)<u(\hat x)+(y-\hat x)\cdot\n\Gm(\hat x)-C
R^{\sigma-2}(f(\hat x)+J_{\sigma}(R)|\n\Gm(\hat x)|)
r^2\}|}{|S_r(\hat x)|}\le\ep,\\
&\int_{\overline Q_j}g_{\xi}(\n\Gm (y))\det [D^2 \Gm(y)]^-\,dy\le C
R^{n(\sigma-2)}\sup_{y\in\overline Q_j}(J_{\sigma}(R)^n+\xi^{-n}
|f(y)|^n)\, |Q^*_j|
\end{split}
\end{equation*}
for any $\e>0$ and a cube $Q_j\subset B_{r/4}(x)$ with diameter
$d_j$ such that $x\in\overline Q_j$ and $r/4<d_j<r/2$. So we easily
see that $\overline Q_j\subset B_{r/2}(\hat x)$ and $B_r(\hat
x)\subset 4\sqrt n\,Q_j$. We recall that $\Gm(y)\le u(\hat
x)+(y-\hat x)\cdot\n\Gm(\hat x)$ for any $y\in B_{2R}$ because $\Gm$
is concave on $B_{2R}$ and $\Gm(\hat x)=u(\hat x)$. Since $d_j$ is
comparable to $r$, it thus follows that
\begin{equation*}\begin{split} &\bigl|\{y\in 4\sqrt
n\,Q_j:u(y)\ge\Gm(y)-C R^{\sigma-2}\sup_{\overline Q_j}(f+J_{\sigma}(R)|\n\Gm|)d_j^2\}\bigr|\\
&\ge\bigl|\{y\in 4\sqrt n\,Q_j:u(y)\ge
u(\hat x)+(y-\hat x)\cdot\n\Gm(\hat x)-C R^{\sigma-2}(f(\hat x)+J_{\sigma}(R)|\n\Gm(\hat x)|)r^2\}\bigr|\\
&\ge(1-\ep)\bigl|S_r(\hat x)\bigr|\ge\xi_0
|Q_j|.\end{split}\end{equation*} Thus we proved (f). Moreover, (e)
holds for $Q_j$ because $\overline Q_j\subset B_{r/2}(\hat x)$ and
$B_r(\hat x)\subset 4\sqrt n\,Q_j$. Hence the cube $Q_j$ would not
split and the process must stop there. \qed

\subsection{A-B-P estimate for a class $\cS^{\fL_0}_{\eta}$ : $0<\sigma < 2$}\label{sec-abp-l}
For $\cI \in \cS^{\fL_0}_{\eta}$, there are $ \cB^{\pm}_R$ as
Definition \ref{def-S-eta}. Set
$$\cC^{\pm}_{\eta}(u,\Gm,B_R)=\{x\in\cC(u,\Gm,B_R):  \cB^{\pm}_R(\n\Gm (x))\leq 0\}$$
for a concave envelope $\Gm $ of $u$ in $B_{2R}$.

\begin{lemma}\label{lem-6.7} Let $0<\sigma<2$, $I\in \cS^{\fL_0}_{\eta}$, and $0<R\le R_0$.
Let $u\le 0$ in $\BR^n\s B_R$ and let $\Gm$ be its concave envelope
in $B_{2R}$. If $\,u\in \rB(\BR^n)$ is a viscosity subsolution to
$\cI u=-f$ on $B_R$ where $f:\BR^n\to\BR$ is a function with $f>0$
on $\cC(u,\Gm,B_R)$, then there exist constants $C>0$ depending only
on $n,\ld$ and $\Ld$ $($but not on $\sm$$)$ such that for any
$x\in\cC^+_{\eta}(u,\Gm,B_R)$ and $M>0$ there is some
$k\in\BN\cup\{0\}$ such that
\begin{equation}\label{eq-6.5}
\bigl|\{y\in R_k: \mu^-(u,x,y;\n\Gm)\geq  M_0 r_k^2\}\bigr|\le
C\f{R^{\sigma-2}(f(x)+R^{1-\sigma}|\n\Gm(x)|)}{M}|R_k|
\end{equation} where $R_k=B_{r_k}\s B_{r_{k+1}}$ for $r_k=\vr_0 2^{-\f{1}{2-\sm}-k}R$ and
$\vr_0=1/(16\sqrt n)$. Here, $\n\Gm(x)$ denotes any element of the
superdifferential $\pa\Gm(x)$ of $\,\Gm$ at $x$.
\end{lemma}

\pf Take any $x\in\cC^+_{\eta}(u,\Gm,B_R)$. From the definition of
$\cS^{\fL_0}_{\eta}$,  we have that
\begin{equation*}\begin{split}
\cI u(x;\n\Gm)&\leq \cM^+_{\fL_0,R} u(x;\n\Gm)+\cB^+_{R}\left(\n \Gm(x)\right )+(2-\sigma)CR^{1-\sigma}|\n\Gm(x)|\\
&\leq \cM^+_{\fL_0,R} u(x;\n\Gm)+(2-\sigma)CR^{1-\sigma}|\n\Gm(x)|
\end{split}\end{equation*}
because $\cB^+_{R}\left(\n \Gm (x)\right )\leq  0$ from the
assumption and $\mu^+_R(u,x,\,\cdot\,;\n\Gm)=0$ on $\BR^n$. The
conclusion comes from similar arguments as Lemma \ref{lem-6.1}.

\qed

Now we have the following Corollary as Section \ref{sec-abp-blow}.
\begin{cor}\label{cor-6.8} Let $\,u$ and $\Gm$ be functions as in
Lemma \ref{lem-6.7}. Then there exist a finite family
$\{Q_j\}_{j=1}^m$ of open cubes $Q_j$ with diameters $d_j$ such that

$(a)$ Any two cubes $Q_i$ and $Q_j$ do not intersect, $(b)$
$\tQ_j=\cC^+_{\eta}(u,\Gm,B_R)\cap\overline Q_j\neq\phi$ for any
$j$,

$(c)$ $\cC^+_{\eta}(u,\Gm,B_R)\subset\bigcup_{j=1}^m\tQ_j$, $(d)$
$d_j\le\vr_0 2^{-\f{1}{2-\sm}}R$ where $\vr_0=1/(16\sqrt n)$,

$(e)$ $\int_{\overline Q_j}g_{\xi}(\n\Gm (y))\det (D^2
\Gm(y))^-\,dy\le C R^{n(\sigma-2)}(\sup_{\overline Q_j}(
R^{n(1-\sigma)}+\xi^{-n}|f|^n) |\tQ_j|$,

$(f)$ $|\{y\in 4\sqrt n\,Q_j:u(y)\ge\Gm(y)-C
R^{(\sigma-2)}(\sup_{\overline Q_j} (f+R^{1-\sigma}|\n\Gm|)
d_j^2\}|\ge\xi_0 |Q_j|$,

\noindent where the constants $C>0$ and $\xi_0>0$ depend on $n,\Ld$
and $\ld$ $($ but not on $\sm$$)$.
\end{cor}
\subsection{Discussion of A-B-P estimates}
At this subsection, we are going to discuss motivations and
differences of the A-B-P estimates at previous subsections.
\begin{remark}
\item
\begin{enumerate}
\item Key setp in  the A-B-P estimate is the control of the volume of
  the gradient image, $|\n \Gamma (B_R)|$,  in terms of $\frac{M_0}{R}$
  for $M_0=\sup_{B_R}u$. From the concavity $\Gamma(x)$, we have
  $B_{\frac{M_0}{R} }\subset\n \Gamma (B_R)=\n \Gamma
  (\cC(u,\Gm,B_R))$ and then $\omega_n\left(
    \frac{M_0}{R} \right)^n\leq |\n\Gamma (\cC(u,\Gm,B_R;b_R))|$ $($see Lemma 9.2, \cite{GT}$)$.
Similarly for $\cB^{\pm}_R\neq 0$, we have $\{z\in B_{\frac{M_0}{R} } :\cB^{\pm}_R(z)\leq 0\}\subset\n \Gamma
  (\cC^+_{\eta}(u,\Gm,B_R))$ and then $\eta\omega_n\left(
    \frac{M_0}{R} \right)^n\leq |\n\Gamma (\cC^+_{\eta}(u,\Gm,B_R))|$.
\item The different A-B-P estimates have been considered to control
  the effect of the gradient term $\n \Gamma (x)$ caused by the fact
  that the Kernel is not symmetric. And they will be used at Lemma
  \ref{lem-7.4} to prove the decay estimate of the upper level set of
  super-solutions.
When we apply A-B-P estimate with blow-up rate, we have an extra term
$(R^{\sigma-2}J_{\sigma})^n|B_R|\approx 1$ (for $1<\sigma<2$), $-\log
(R)$ (for $\sigma=1$), and $R^{\sigma-1}$ (for $0<\sigma<1$) caused by
$\n \Gamma (x)$.
It is bounded only at $1<\sigma<2$. This is the main reason that we
another  A-B-P estimates.
\item For the operators in $\cS_{\eta}$,
 $(R^{\sigma-2}R^{1-\sigma})^n|B_R|\approx 1$ (for $0<\sigma<2$) gives us the decay estimate, Lemma \ref{lem-7.4}.
 \item For $1<\sigma<2$, Lemma \ref{lem-6.1} and Lemma \ref{lem-6.7} are equal since $J_{\sigma}(R)\approx R^{1-\sigma}$ for $0<R\leq 1$.
\end{enumerate}
\end{remark}
\section{Decay Estimate of Upper Level Sets}\label{sec-7}

In this section, we are going to show the geometric decay rate of the
upper level set of nonnegative solution $u$. The key Lemma \ref{lem-7.4}
says that if  a nonnegative function $u$ has a value smaller than one
in $Q_R$
then the lower level set $\{x: u\leq M\}$ has uniformly positive amount of
measure $\nu |Q_R|$ which will be proven through ABP estimate.
But the assumption of ABP estimate on a subsolution  requires its
special shape: it should be negative out side of $Q_R$ and positive at some
interior point. So we are going to construct a special function $\Psi$
so that $\Psi-u$ meets the requirement of ABP estimate.

\subsection{Special functions}

The construction of the special function is based on the idea in
\cite{CS, KL}. Nontrivial finer computation has been done to take care
of the influence of the gradient term and the lack of scaling.

\begin{lemma}\label{lem-7.1} There exist some
$\sm^*\in (0,2)$ and $p>0$ such that the function
$$f(x)=\min\{2^pR^{-p},|x|^{-p}\}$$ is a subsolution to $\cM^-_{\fL_0} f(x)\ge 0
$ for any $\sm\in (\sm^*,2)$ and $x\in B_R^c$.\end{lemma}

\pf It is enough to show that there is some $\sm^*\in(1,2)$ so that
\begin{equation}\label{eq-7.1}
\cM^-_{\fL_0} f(x)\ge 0
\end{equation}
for $x=R_1e_n=(0,0,\cdots,0,R_1)\in\BR^n$; for every other $x$ with
$|x|=R_1\geq R$, the above inequality follows by rotation. In order
to prove \eqref{eq-7.1}, we use the following elementary inequality
that holds for any $a>b>0$ and $p>0$; $$(a+b)^{-p}\ge
a^{-p}\biggl(1-p\,\f{b}{a}+\frac{p(p+1)}{2!}\bigl(\frac{b}{a}\bigr)^2
-\frac{p(p+1)(p+2)}{3!}\bigl(\frac{b}{a}\bigr)^3\biggr).$$ Using
this inequality and $\mu(f,R_1e_n,y)=R_1^{-p}
\mu\left(f,e_n,\overline y\right) $ for $\overline{y}=y/R_1$, we
have that
\begin{equation}\label{eq-7.4}
\begin{split}\mu(f,e_n,\overline{y})&=|e_n+\overline{y}|^{-p}-1+p\,\overline{y}_n=(1+|\overline{y}|^2+2 \overline{y}_n)^{-p/2}-1+p\,\overline{y}_n\\
&\ge-\left(\f{p}{2}+1\right)|\overline{y}|^2+\f{p(p+2)}{2}\,\overline{y}_n |\overline{y}|^2+\f{|\overline{y}|^2}{(1+|\overline{y}|^2)^{p/2+1}}\\
&+\f{p(p+2)}{2}\f{\overline{y}_n^2}{(1+|\overline{y}|^2)^{p/2+2}}
-\f{p(p+2)(p+4)}{6}\f{\overline{y}_n^3}{(1+|\overline{y}|^2)^{p/2+3}} \\
\end{split}
\end{equation}
for any $y\in B_{\f12 R}$. We choose some sufficiently large $p>0$
so that
\begin{equation}\label{eq-7.3}
\f{p(p+2)}{2(1+r^2)^{p/2+2}}\int_{S^{n-1}} \theta_n^2 \,
d\sigma(\theta)+\f{\om_n}{(1+r^2)^{p/2+1}}-\biggl(\f{p}{2}+1\biggr)\om_n=\delta_0(r)>0
\end{equation} for any sufficiently small $r>0$.
Since
$\int_{S^{n-1}}\theta_n\,d\sm(\theta)=\int_{S^{n-1}}\theta_n^3\,d\sm(\theta)=0$,
it follows from \eqref{eq-3.1}, \eqref{eq-3.2}, \eqref{eq-7.4} and
\eqref{eq-7.3} that
\begin{equation*}
\begin{split}&\cM^-_{\fL_0}
f(e_n)\\&\ge(2-\sm)R_1^{-p}\biggl(\int_{\BR^n}\f{\ld\mu^+(f,e_n,y/R_1)}{|y|^{n+\sm}}\,dy-\int_{\BR^n}\f{\Ld\mu^-(f,e_n,y/R_1)}{|y|^{n+\sm}}\,dy\biggr)\\
&\ge(2-\sm)
R_1^{-p-\sigma}\biggl(\ld\,\int_{B_r}\f{\mu(f,e_n,y)}{|y|^{n+\sm}}\,dy
-\Ld\,\int_{\BR^n\s B_r}\f{\mu^-(f,e_n,y)}{|y|^{n+\sm}}\,dy\biggr)\\
&\ge(2-\sm) R_1^{-p-\sigma}\biggl(\f{\ld\dt_0(r)}{2-\sm}
-(2^p+1)\Ld\int_{\BR^n\s B_r}\f{1}{|y|^{n+\sm}}\,dy-p\Ld\int_{\BR^n\s B_r}\f{1}{|y|^{n-1+\sm}}\,dy\biggr)\\
&=R_1^{-p-\sigma}\biggl(\ld\delta_0(r)-(2^p+1)\Ld\om_n\f{2-\sm}{\sm}
\,r^{-\sm}-p\Ld\om_n\f{2-\sm}{\sm}
\,r^{1-\sm}\biggr)
\end{split}
\end{equation*}
 for $r\in (0,\f12R)$, where $\om_n$
denotes the surface measure of $S^{n-1}$. Thus we may take some
sufficiently small $r\in (0,1/2)$ and take some $\sm^*\in (1,2)$
close enough to $2$ in the above so that
$$\cM^-_{\fL_0} f(e_n)\ge 0$$ for any $\sm\in
(\sm^*,2)$. Hence we complete the proof. \qed

Now we have the following Corollary as Corollary 4.1.2, \cite{KL}. The only difference is that the influence of non-symmetry of the kernel is $p|y|\chi_{B_1}(y)$, not $p|y|$ in the proof since authors considered a larger class of operators for $1<\sigma<2$ at \cite{KL}.
\begin{cor}\label{cor-7.2}   Given any $\sm_0\in(1,2)$, there exist some $\dt>0$ and $p>0$ such that the function
$$f(x)=\min\{\delta^{-p}R^{-p},|x|^{-p}\}$$ is a subsolution to $\cM^-_{\fL_0} f(x)\ge 0
$ for any $\sm\in (\sm_0,2)$ and $x\in B_{R/4}^c$.
\end{cor}

\begin{lemma}\label{lem-7.2-2}   Given any $\sm\in(0,1]$, there exist $0<R_0<1$, $\dt>0$, $p>0$, and $\ep_0>0$  such that the function
$$f(x)=\exp (-p|x|^{\sigma/4})$$ is a subsolution to $\cM^-_{\fL_0} f(x)\ge 0
$ for any  $0<R<R_0$, and $x\in B_{2\sqrt{n}R}\backslash B_{R/4}$.
\end{lemma}

\begin{proof}
We consider $f(x)=\exp (-p|x|^{\alpha})$ for $\alpha>0$ to be
selected later. For any $x\in B_{2\sqrt{n}R}\backslash B_{R/4}$, we
may assume $x=R_1e_n$ for $R/4\le R_1<2\sqrt{n}R$ without losing
generality. Then we have
\begin{equation}
\begin{split}
\mu(f,R_1e_n,y)&=e^{-p|R_1e_n+y|^{\alpha}}-e^{-pR_1^{\alpha}}+\alpha pR_1^{\alpha-1}e^{-pR_1^{\alpha}}y_n\chi_{|y|<1}\\
&=e^{-p(|y|^2+R_1^2+2y_n)^{\alpha/2}}-e^{-pR_1^{\alpha}}+\alpha pR_1^{\alpha-1}e^{-pR_1^{\alpha}}y_n\chi_{|y|<1}\\
\end{split}
\end{equation}
and
\begin{equation*}\begin{split}\cM^-_{\fL_0}
f(R_1e_n)&=(2-\sm)\ld\int_{\BR^n}\f{\mu^+(f,R_1e_n,y)}{|y|^{n+\sm}}\,dy
-(2-\sm)\Ld\int_{\BR^n}\f{\mu^-(f,R_1e_n,y)}{|y|^{n+\sm}}\,dy\\
&\fd\cJ_1 f(R_1e_n)+\cJ_2 f(R_1e_n).\end{split}\end{equation*} We
note that $\mu^-(f,R_1e_n,y)=0$ for any $y\in B_{\rho( R_1)}$ and
$\rho(R_1)\geq  c_1 R_1$ for a uniform constant $c_1>0$ from simple
geometric observation. It is easy to check that
$\mu^-(f,R_1e_n,y)\le  e^{-pc_1^{\alpha}R_1^{\alpha}} +\frac{\alpha
p e^{-pc_1^{\alpha}R_1^{\alpha}} }{R_1^{1-\alpha}}|y|\chi_{B_1}(y)$
for any $y\in B_{\rho(R_1)}^c$ . So we see that
\begin{equation*}
\begin{split}
-\cJ_2 f(R_1e_n)&=(2-\sm)\Ld\int_{|y|\ge
\rho(R_1)}\f{\mu^-(f_{\dt},R_1e_n,y)}{|y|^{n+\sm}}\,dy\\
&\le(2-\sm_0)\Ld e^{-pc_1^{\alpha}R_1^{\alpha}}\int_{|y|\ge
\rho(R_1)  }\f{1+\frac{\alpha p}{R_1^{1-\alpha}}|y|\chi_{B_1}(y)}{|y|^{n+\sm}}\,dy.\\
&\leq  (2-\sm_0)\Ld e^{-pc_1^{\alpha}R_1^{\alpha}}\left(\frac{c_1}{R_1^{\sigma}}+\frac{\alpha pc_2}{R_1^{\sigma-\alpha}}\right).
\end{split}
\end{equation*}
Also we have that
\begin{equation*}
\begin{split}
\cJ_1 f(R_1e_n)&\ge(2-\sm)\ld\int_{\delta R_1/2<|y+Re_n|<\dt R_1}\f{ e^{-p\delta^{\alpha}R_1^{\alpha}}-e^{-pR_1^{\alpha}} -\alpha p R_1^{\alpha}e^{-pc_1^{\alpha}R_1^{\alpha}}  }{|y|^{n+\sm}}\,dy\\
&\quad\ge \frac{c_3(2-\sm)\ld}{R_1^{\sigma}} e^{-p\delta^{\alpha}R_1^{\alpha}}\quad\text{for a uniformly small $\delta>0$ and $c_3>0$.}
\end{split}
\end{equation*}
Set $\alpha=\sigma/4$ for a uniformly small $\ep_0$. There are a  large $p>0$ and small $R_0,\delta>0$ such that  if $0<R<R_0$, then we have $e^{-p c_1^{\alpha}R_1^{\alpha}}<< e^{-p\delta^{\alpha}R_1^{\alpha}}$ and
\begin{equation*}
\begin{split}
-\cJ_2 f(R_1e_n)&\leq  (2-\sm_0) e^{-pc_1^{\alpha}R_1^{\alpha}}\left(\frac{c_1}{R_1^{\sigma}}+\frac{\alpha pc_2}{R_1^{3\sigma/4}}\right)<<\cJ_1 f(R_1e_n)
\end{split}
\end{equation*}
for a uniformly small $\delta>0$ and $R/4\le R_1<\sqrt{n}R$.
\end{proof}

\begin{lemma}\label{lem-7.3} Given any $\sm_0\in (1,2)$ and $0<R<R_0$,  there exists a function $\Psi\in\rB(\BR^n)$ such that

$(a)$ $\Psi$ is continuous on $\BR^n$, $(b)$ $\Psi=0$ on
$B^c_{2\sqrt{n} R}\,$, $(c)$ $\Psi>2$ on $Q_{3R}$,

$(d)$ $\Psi\le M$ on $\BR^n$ for some $M>1$, $(e)$
$\cM^-_{\fL_0}\Psi$ is continuous on $B_{2\sqrt{n}R}\,$,

$(f)$ $\cM^-_{\fL_0}\Psi>-\psi/R^{\sigma}$ on $\BR^n$ where $\psi$
is a positive bounded function on $\BR^n$ which is supported in
$\overline B_{R/4}\,$, for any $\sm\in (\sm_0,2)$.\end{lemma}

\pf Let $f(x)=\min\{\delta^{-p}R^{-p},|x|^{-p}\}$. Then we consider
the function $\Psi$ given by
$$\Psi=c\,
\begin{cases} 0 &\text{ in $\BR^n\s B_{2\sqrt n R}$,}\\
f(|x|)- f(2\sqrt n R)&\text{ in $B_{2\sqrt n R}\s B_{R/4}$,}\\
P &\text{ in $B_{R/4}$,}
\end{cases}
$$
where $P$ is a quadratic paraboloid chosen so that $\Psi$ is
$\rC^{1,1}$ across $\pa B_{R/4}$. We now choose the constants $c$
and $\dt$ so that $\Psi(x)>2$ for $x\in Q_{3R}$ and $\Psi\le M$ on
$\BR^n$ for some $M>1$ (recall that $Q_{3R}\subset B_{3\sqrt n
R/2}\subset B_{2\sqrt n R}$). Since $\Psi\in\rC^{1,1}(B_{2\sqrt n
R})$, $\cM^-_{\fL_0}\Psi$ is continuous on $B_{2\sqrt n R}$. Also by
Lemma \ref{lem-7.1} we see that $\cM^-_{\fL_0}\Psi\ge 0$ on
$B^c_{R/4}$. Hence this completes the proof. \qed

\begin{lemma}\label{lem-7.3-2} Given any $\sm\in (0,1]$ and $0<R<R_0$,  there exist a function $\Psi\in\rB(\BR^n)$ such that

$(a)$ $\Psi$ is continuous on $\BR^n$, $(b)$ $\Psi=0$ on
$B^c_{2\sqrt{n} R}\,$, $(c)$ $\Psi>2$ on $Q_{3R}$,

$(d)$ $\Psi\le M$ on $\BR^n$ for some $M>1$, $(e)$
$\cM^-_{\fL_0}\Psi$ is continuous on $B_{2\sqrt{n}R}\,$,

$(f)$ $\cM^-_{\fL_0}\Psi>-\psi/R^{\sigma}$ on $\BR^n$ where $\psi$
is a positive bounded function on $\BR^n$ which is supported in
$\overline B_{R/4}\,$.\end{lemma}

\pf Let $f(x)=c\,\exp (-p|x|^{\sigma/4})$ as in Lemma
\ref{lem-7.2-2}.
 From Lemma \ref{lem-7.2-2}, we know that $\cM^-_{\fL_0}\bigl(f(|x|)-f(2\sqrt{n}R)\bigr)\ge 0$ on $B_{2\sqrt{n}R}\backslash B_{R/4}$.
 And the zero function is a solution of $\cM^-_{\fL_0} u=0$ on  $B^c_{R/4}$. Since the suprimum of two solutions is a subsolutions, we have
 $$\cM^-_{\fL_0}\left[\max\bigl\{0,f(|x|)-f(2\sqrt{n}R)\bigr\}\right]\ge 0\,\,\text{ on
 $B^c_{R/4}$.}$$
The conclusion comes from the  construction of  $\Psi$ as Lemma \ref{lem-7.3}.
 \qed

\subsection{Estimates in measure}

The main tool that shall be useful in proving H\"older estimates
is a lemma that connects a pointwise estimate with an estimate in
measure. The corresponding lemma in our context is the following.

\begin{lemma}\label{lem-7.4} Let $\sm_0\in (1,2)$ and assume that  $\,\sm\in
(0,1]$ or $\,\sm\in
(\sm_0,2)$.  If $R\in(0,R_0]$, then there exist some constants
$\vep_0>0$, $\nu\in (0,1)$ and $M>1$  for which if
$\,u\in\rB(\BR^n)$ is a viscosity supersolution to either $\cM^-_{\fL_0}
u\le\vep_0/R^{\sigma}$ with $\sigma> 1$  or $\cM^-_{\fL_0,R,\eta}
u\le\vep_0/R^{\sigma}$ with $\sigma\leq  1$,  $u\geq 0$ on $Q_{4\sqrt{n}R}$, and $\inf_{Q_{3R}}u\le 1$, then
 $\,|\{u\le M\}\cap Q_R|\ge\nu |Q_R|$.
 For $\sm\in (\sm_0,2)$, $\vep_0$, $\nu$ and $M$ depend only on $\ld,\Ld$ , the dimension $n$, and $\sm_0$. And for $\,\sm\in
(0,1]$ , $\vep_0$, $\nu$ and $M$ depend only on $\ld,\Ld$ , $n$, $\sm$, and $\eta$.
\end{lemma}

\noindent{\it Remark.}  We denote by $Q_r(x)$ an open cube
$\{y\in\BR^n:|y-x|_{\iy}\le r/2\}$ and $Q_r=Q_r(0)$. If we set
$Q=Q_r(x)$, then we denote by $s Q=Q_{s r}(x)$ for $s>0$.

{\it $[$Proof of Lemma \ref{lem-7.4}$]$}
(Case {\bf 1}: $1<\sigma_0<\sigma<2$.)
 We consider the function $v:=\Psi-u$ where
$\Psi$ is the special function constructed in Lemma \ref{lem-7.3}.
Then we easily see that $v$ is upper semicontinuous on $\overline
B_{2\sqrt n R}$ and $v$ is not positive on $\BR^n\s B_{2\sqrt n R}$.
Moreover, $v$ is a viscosity subsolution to $\cM^+_{\fL_0}
v\ge\cM^-_{\fL_0}\Psi -\cM^-_{\fL_0} u\ge-(\psi+\vep_0)/R^{\sm}$ on
$B_{2\sqrt n R}$. So we want to apply Theorem \ref{thm-6.4}
(rescaled) to $v$. Let $\Gm$ be the concave envelope of $v$ in
$B_{4\sqrt n R}$. Since $\inf_{Q_R}u\le 1$, $\inf_{Q_R}\Psi>2$ and
$Q_R\subset B_{2\sqrt n R}$, we easily see that
$M_0:=\sup_{B_{2\sqrt n R}}v=v(x_0)>1$ for some $x_0\in B_{2\sqrt n
R}$. We consider the function $g$ whose graph is the cone in
$\BR^n\times\BR$ with vertex $(x_0,M_0)$ and base $\pa B_{6\sqrt n
R}(x_0)\times\{0\}$. For any $\xi\in\BR^n$ with $|\xi|<M_0/6\sqrt n
R$, the hyperplane
$$H=\{(x,x_{n+1})\in\BR^n\times\BR:x_{n+1}=L(x):=
M_0+\xi\cdot(x-x_0)\}$$ is a supporting hyperplane for $g$ at $x_0$
in $B_{6\sqrt n R}(x_0)$. Then $H$ has a parallel hyperplane $H'$
which is a supporting hyperplane for $v$ in $B_{4\sqrt n R}$ at some
point $x_1\in B_{2\sqrt n R}$. By the definition of concave envelope,
we see that $H'$ is also the hyperplane tangent to the graph of
$\Gm$ at $x_1$, so that $\xi=\n\Gm(x_1)$. This implies that
$B_{M_0/(6\sqrt n R)}(0)\subset\n\Gm(B_{2\sqrt n R})$. Thus we have that
\begin{equation}\label{eq-7.6}
C(n)\ln \left(\frac{(M_0/R)^n}{\xi^n}+1\right)\leq
\int_{\cC(u,\Gm,B_R) }g_{\xi}(\n\Gm (y)) \det[D^2 \Gm(y)]^-\,dy,
\end{equation}
 where $g_{\xi}$ is the function
given in Corollary \ref{cor-6.3}. We also observe as shown in \cite{CC} that
\begin{equation}\label{eq-7.7}
\bigl|\n\Gm\bigl(B_{2\sqrt n R}\s\cC(v,\Gm,B_{2\sqrt
n R} )\bigr)\bigr|=0.
\end{equation}
 Let $\{Q_j\}$ be the finite family of
cubes given by Theorem \ref{thm-6.4} (rescaled on $B_{2\sqrt n R}$). Then it
follows from \eqref{eq-7.6}, \eqref{eq-7.7} and Theorem \ref{thm-6.4} that, for $Q^*_j=\overline{Q}_j\cap \cC(u,\Gm,B_{2\sqrt{n}R})$,
\begin{equation}\label{eq-7.8}\begin{split}&\ln\biggl(\frac{[(\sup_{B_{2\sqrt
          n R}}v)/R]^n}{\xi^n}+1\biggr)\le
C\,\int_{\cC(u,\Gm,B_{2\sqrt{n}R}) }g_{\xi}(\n\Gm (y))\det[D^2\Gm(y)]^-\,dy\\
&\qquad\qquad\le C\biggl(\sum_j\sup_{\overline
Q_j}\bigl((R^{\sigma-2} J_{\sigma}(R))^n+\xi^{-n}(R^{(\sigma-2)} (\psi+\vep_0)/R^{\sigma}\bigr)^n |Q^*_j|\biggr)\\
&\qquad\qquad\le C\biggl((R^{\sigma-2}
J_{\sigma}(R))^n\sum_j|Q^*_j|+\xi^{-n}\sum_j\sup_{\overline
Q_j}((\psi+\vep_0)/R^{2})^n |Q^*_j|\biggr)\\
&\qquad\qquad\le C\biggl((R^{\sigma-1}
J_{\sigma}(R))^n+\xi^{-n}\sum_j\sup_{\overline
Q_j}((\psi+\vep_0)/R^{2})^n |Q^*_j|\biggr).
\end{split}
\end{equation}
Here we note that $K_R:=Exp((R^{\sigma-1} J_{\sigma}(R))^n)\leq
C<\infty$ for $1<\sigma<2$ and $R<1$. If we set
$\xi=\bigl(\sum_j\sup_{\overline Q_j}((\psi+\vep_0)/R^{2})^n
|Q^*_j|\bigr)^{1/n}$ in \eqref{eq-7.8}, then we have that
\begin{equation}\label{eq-7.9}
\begin{split}\sup_{B_{2\sqrt n R}}v&\le CR\biggl(\sum_j\sup_{\overline Q_j}((\psi+\vep_0)/R^{2})^n
|Q^*_j|\biggr)^{1/n}\\
&\le C
\vep_0+CR\biggl(\sum_j\bigl((\sup_{\overline
Q_j}\psi)/R^{2}\bigr)^n |Q^*_j|\biggr)^{1/n}.\end{split}
\end{equation}
Since $\inf_{Q_R}u\le 1$ and $\inf_{Q_R}\Psi>2$, we see that
$\sup_{B_{2\sqrt n R}}v>1$. If we choose $\vep_0$ small enough, the
above inequality \eqref{eq-7.9} implies that
$$\f{1}{2^{1/n}} R\le C\biggl(\sum_j\bigl(\sup_{Q_j}\psi\bigr)^n
|Q^*_j|\biggr)^{1/n}.$$ We recall from the proof of Lemma
\ref{lem-7.3} that $\psi$ is supported on $\overline B_{R/4}$ and
bounded on $\BR^n$. Thus the above inequality becomes $$\f{1}{2}
|Q_R|\le C\biggl(\,\sum_{Q_j\cap B_{R/4}\neq\phi}
|Q^*_j|\,\biggr),$$ which provides a lower bound for the sum of the
volumes of the cubes $Q_j$ intersecting $B_{R/4}$ as follows;
\begin{equation}\label{eq-7.10}\sum_{Q^*_j\cap
B_{R/4}\neq\phi} |Q^*_j|\ge \nu |Q_R|.
\end{equation}
Since $y\in\cC(v,\Gm,B_{2\sqrt nR})$ implies $v(y)\ge 0$, and thus
$u(y)\le\Psi(y)\le M$ by Lemma \ref{lem-7.3}. Hence it follows from
\eqref{eq-7.10} that
\begin{equation*}|\{u\le M\}\cap Q_R|\ge|\{u\le\Psi\}\cap
Q_R|\ge|\cC(v,\Gm,B_{2\sqrt n R})\cap Q_R|\ge\sum_{Q_j^*\cap
B_{R/4}\neq\phi} |Q_j^*|\ge\nu |Q_R|.
\end{equation*} Hence we complete the proof of (Case 1).

(Case {\bf 2}: $0<\sigma\leq 1$.)
We know
$B_{M_0/(6\sqrt n R)}(0)\subset\n\Gm(B_{2\sqrt{n} R})$ and
$|\n\Gm(B_{ 2\sqrt{n}R})\cap \cC^+_{\eta}(u,\Gm,B_{2\sqrt{n}R}))|\geq \eta |B_{M_0/(6\sqrt n R)}(0)|$

Thus we have that, $Q^*_j=\overline{Q}_j\cap \cC_{\eta}(u,\Gm,B_{2\sqrt{n}R})$
\begin{equation}\label{eq-7.6}
 C(n)\ln \left(\frac{\eta(M_0/R)^n}{\xi^n}+1\right)\leq \int_{\cC^+_{\eta}(u,\Gm,B_{2\sqrt{n}R}) }g_{\xi}(\n\Gm (y))
\det[D^2 \Gm(y)]^-\,dy,
\end{equation}
 where $g_{\xi}$ is the function
given in Corollary \ref{cor-6.3}.
 Let $\{Q_j\}$ be the finite family of
cubes given by Theorem \ref{thm-6.4} (rescaled on $B_{2\sqrt n R}$). Then it
follows from \eqref{eq-7.6}, \eqref{eq-7.7} and Theorem \ref{thm-6.4} that
\begin{equation}\label{eq-7.8}\begin{split}&\ln\biggl(\frac{\e[(\sup_{B_{2\sqrt
          n R}}v)/R]^n}{\xi^n}+1\biggr)\le
C\,\int_{\cC(u,\Gm,B_{2\sqrt{n}R}) }g_{\xi}(\n\Gm (y))\det[D^2\Gm(y)]^-\,dy\\
&\qquad\qquad\le C\biggl(\sum_j\sup_{\overline
Q_j}\bigl((R^{\sigma-2} R^{1-\sigma})^n+\eta^{-n}(R^{(\sigma-2)} (\psi+\vep_0)/R^{\sigma}\bigr)^n |Q^*_j|\biggr)\\
&\qquad\qquad\le C\biggl((R^{-1})^n\sum_j|Q_j|+\xi^{-n}\sum_j\sup_{\overline
Q_j}((\psi+\vep_0)/R^{2})^n |Q^*_j|\biggr)\\
&\qquad\qquad\le C\biggl(1+\xi^{-n}\sum_j\sup_{\overline
Q_j}((\psi+\vep_0)/R^{2})^n |Q^*_j|\biggr).
\end{split}
\end{equation}
 The conclusion comes from the similar argument as (Case 1).\qed

We split  $Q_R$ into $2^n$ cubes of
half side. We do the same splitting step with each one of these
$2^n$ cubes and we continue this process. The cubes obtained in this
way are called {\it dyadic cubes.} If $Q$ is a dyadic cube different
from $Q_R$, then we say that $\Qtil$ is the {\it predecessor} of $Q$
if $Q$ is one of $2^n$ cubes obtained from splitting $\Qtil$.

\begin{lemma}\cite{CC}\label{lem-7.5} Let $A,B$ be measurable sets with $A\subset
B\subset Q_R$. If $\,\dt\in (0,1)$ is some number such that $(a)$
$|A|\le\dt$ and $(b)$ $\Qtil\subset B$ for any dyadic cube $Q$ with
$|A\cap Q|>\dt|Q|$, then $|A|\le\dt |B|$.\end{lemma}

The following lemma is a consequence of Lemma \ref{lem-7.4} and
Lemma \ref{lem-7.5}.

\begin{lemma}\label{lem-7.6} Under the same condition as Lemma \ref{lem-7.4}, there are universal constants $C>0$ and $\vep_*>0$ such
that
$$\bigl|\{u>t\}\cap Q_R\bigr|\le C\,t^{-\vep_*}|Q_R|,\,\forall\,t>0.$$
\end{lemma}

\begin{remark}
 We note that $B_{R/2}\subset Q_R\subset Q_{3R}\subset
B_{3\sqrt{n}R/2}\subset B_{2\sqrt{n}R}.$
\end{remark}

\pf
(Case {\bf 1}: $1<\sigma_0<\sigma<2$.)

  First, we shall prove that \begin{equation}\label{eq-7.12}\bigl|\{u>M^k\}\cap
Q_R\bigr|\le(1-\nu)^k|Q_R|,\,\forall\,k\in\BN,\end{equation} where $\nu>0$
is the constant as in Lemma \ref{lem-7.4} and $M>1$ is the constant chosen in
Lemma \ref{lem-7.4}

If $k=1$, then it has been done in Lemma \ref{lem-7.4}. Assume that the result
\eqref{eq-7.12} holds for $k-1$ ($k\ge 2$) and let $$A=\{u>M^k\}\cap
Q_R\,\,\text{ and }\,\,B=\{u>M^{k-1}\}\cap Q_R.$$ If we can show
that $|A|\le(1-\nu)|B|$, then \eqref{eq-7.12} can be obtained for $k$. To
show this, we apply Lemma \ref{lem-7.6}. By Lemma \ref{lem-7.4}, it is clear that
$A\subset B\subset Q_R$ and $|A|\le|\{u>M\}\cap Q_R|\le (1-\nu)|Q_R| $. So
it remain only to prove (b) of Lemma \ref{lem-7.5}; that is, we need to show
that if $Q=Q_{2^{-i}R}(x_0)$ is a dyadic cube satisfying
\begin{equation}\label{eq-7.13}|A\cap
Q|>(1-\nu)|Q|
\end{equation}
 then $\Qtil\subset B$. Indeed, we suppose
that $\Qtil\not\subset B$ and take $x_*\in\Qtil$ such that
\begin{equation}\label{eq-7.14}u(x_*)\le M^{k-1}.
\end{equation}
 We now consider the transformation
$x=x_0+y,\,\,y\in Q_{2^{-i}R},\,\,x\in Q=Q_{2^{-i}R}(x_0)$ and the
function $v(y)=u(x)/M^{k-1}$. If we can show that $v$ satisfies the
hypothesis of Lemma \ref{lem-7.4}, then we have that $\nu |Q|<|\{u(x)\le M^k\}\cap Q|$, and thus $|Q\s A|>\nu
|Q|$ which contradicts \eqref{eq-7.13}.

To complete the proof, we consider once again the transformation
$$x=x_0+z,\,\,z\in B_{\f{\sqrt n}{2^i} R},\,\,x\in B_{\f{\sqrt n}{2^{i-1}} R}(x_0)\subset B_{2\sqrt{n}R}$$ and the
function $v(z)=u(x)/M^{k-1}$. It now remains to show that $v$
satisfies the hypothesis of Lemma \ref{lem-7.4}. We now take any
$\vp\in\rC^2_{2\sqrt{n}2^{-i}R}(v;z)^-$. If we set
$\psi=M^{k-1}\vp(\,\cdot\,-x_0)$, then we observe that
$$\vp\in\rC^2_{B_{2\sqrt{n} 2^{-i}R}}(v;z)^-\,\,\,\Leftrightarrow\,\,\,\psi\in
\rC^2_{B_{2\sqrt{n}2^{-i}R}(x_0)}(u;x_0+z)^-.$$ Since $B_{2\sqrt
n2^{-i}R}(x_0)\subset B_{2\sqrt n R}$, we have that
\begin{equation*}\begin{split}\cM^-_{\fL_0}v(z;\n\vp)&\le\cL
v(z;\n\vp)\\
&=\f{1}{M^{k-1}}\int_{\BR^n}\mu (u,x_0+z,y;\n\psi)K(y)\,dy\\
&:=\f{1}{M^{k-1}}\cL u(x_0+z;\n\psi)\end{split}\end{equation*} for
any $\cL\in\fL$. Taking the infimum of the right-hand side in the
above inequality, we obtain that
$$\cM^-_{\fL_0}v(z;\n\vp)\le\f{1}{M^{k-1}}\cM^-_{\fL_0} u(x_0+z;\n\psi).$$
Thus we have that $$\cM^-_{\fL_0}v(z;\n\vp)\le\f{\vep_0}{(2^{-i}
R)^{\sigma}},$$ because $\cM^-_{\fL_0}u\le\f{\vep_0}{R^{\sigma}}$ on
$B_{2\sqrt{n}R}\,$. Also it is obvious that $v\ge 0$ on $\BR^n$ and
we see from \eqref{eq-7.14} that $\inf_{Q} v\le 1$. Finally the
result follows immediately from \eqref{eq-7.12} by taking
$C=(1-\nu)^{-1}$ and $\vep_*>0$ so that $1-\nu=M^{-\vep_*}$.

(Case {\bf 2}: $0<\sigma\leq 1$.)
The argument above is based on Lemma \ref{lem-7.4}, C-Z decomposition Lemma \ref{lem-7.5}, and the invarinace under translation that $\cM^-_{\fL_0,R,\eta}v(z;\n\vp)$ also has.
A similar argument gives us the conclusion.
\qed

By a standard covering argument we obtain the following theorem.

\begin{thm}\label{thm-7.8}Under the same condition as Lemma \ref{lem-7.4}, then there are universal constants $C>0$ and
$\vep_*>0$ such that
$$\bigl|\{u>t\}\cap B_R\bigr|\le C\,t^{-\vep_*} |B_R|,\,\forall\,t>0.$$\end{thm}

In contrast to symmetric cases, we note that we can not obtain the
following theorem by rescaling the above theorem because our cases
are not scaling invariant. We note that Theorem \ref{thm-7.9} on
$r\in (0,1)$ shall be applied to obtain a Harnack inequality, and
also Theorem \ref{thm-7.9} on $r\in[R,2R]$ will be used to prove
H\"older estimates and an interior $\rC^{1,\ap}$-regularity.

\begin{thm}\label{thm-7.9}Let $\sm_0\in (1,2)$ and assume that  $\,\sm\in
(0,1]$ or $\,\sm\in
(\sm_0,2]$ and that  $R\in(0,R_0]$. For a  constant
$\vep_0>0$ given at Theorem \ref{thm-7.9},  if  $u$ is a viscosity supersolution to either $\cM^-_{\fL_0}
u\le\vep_0/R^{\sigma}$ with $\sigma> 1$  or $\cM^-_{\fL_0,R,\eta}
u\le\vep_0/R^{\sigma}$ with $\sigma\leq  1$  such that  $u\geq 0$ on $\R^n$, then there are
universal constants $\vep_*>0$ and $C>0$ such that
$$\bigl|\{u>t\}\cap B_r(x)\bigr|\le C\,r^n\bigl(u(x)+c_0R^{\sigma}\,r^{\sm}\bigr)^{\vep_*}t^{-\vep_*},\,\forall\,t>0.$$
\end{thm}

\pf
(Case 1: $1<\sigma_0<\sigma<2$.)
Let $x\in\BR^n$ and set $v(z)=u(z+x)/q$ for $z\in B_{2r}$ where
$q=u(x)+c_0 r^{\sm}/\vep_0.$ Take any $\vp\in
\rC^2_{B_{2r}}(v;z)^-$. If we set $\psi=q\,\vp(\,\cdot-x)$, then we
see that $\psi\in\rC^2_{B_{2r}(x)}(u;z+x)$. Thus by the change of
variables we have that
\begin{equation*}\begin{split}\cM^-_{\fL_0}v(z;\n\vp)&\le\cL
v(z;\n\vp)\\
&=\f{1}{q}\int_{\BR^n}\mu( z+x,y;\n\psi)\,K(y)\,dy\\
&:=\f{1}{q}\cL u(z+x;\n\psi)\end{split}\end{equation*} for any
$\cL\in\fL_0$. Taking the infimum of the right-hand side in the
above inequality, we get that
$$\cM^-_{\fL_0}v(z;\n\vp)\le\f{1}{q}\cM^-_{\fL_0}
u(rz+x;\n\psi)\le \f{\vep_0}{r^2}.$$ Thus we have that
$\cM^-_{\fL_0}v\le\f{\vep_0}{r^2}$ on $B_{2r}$. Applying Theorem
\ref{thm-7.8} to the function $v$, we complete the proof.

(Case {\bf 2}: $0<\sigma\leq 1$.)
A similar argument gives us the conclusion.
  \qed
\section{Regularity Theory}\label{sec-8.0}
\subsection{ Harnack inequality}\label{sec-8}

\begin{thm}\label{thm-8.1}
 Let $\sm_0\in (1,2)$ and assume that  $\,\sm\in
(0,1]$ or $\,\sm\in
(\sm_0,2]$ an that $R\in(0,R_0]$.  If $\,u\in\rB(\BR^n)$ is a positive function such
that
$$\cM_{\fL_0}^- u\le \f{C_0}{R^\sigma}\,\,\,\text{ and }\,\,\,\cM_{\fL_0}^+ u\ge -\f{C_0}{R^\sigma}\,\,\text{
with $\sigma_0<\sigma<2$
on $B_{2R}$}$$
or
$$\cM_{\fL_0,R,\eta}^- u\le \f{C_0}{R^\sigma}\,\,\,\text{ and }\,\,\,\cM_{\fL_0,R,\eta}^+ u\ge -\f{C_0}{R^\sigma}\,\,\text{
with $0<\sigma \leq 1$
on $B_{2R}$}$$
 in the viscosity sense, then there is uniform  constant
$C>0$ such that
$$\sup_{B_{R/2}} u\le C\,\bigl(\,\inf_{B_{R/2}}u+C_0 \bigr).$$
For $\sm\in (\sm_0,2)$, $C$ depends only on $\ld,\Ld$ , the dimension $n$, and $\sm_0$. And for $\,\sm\in
(0,1]$ , $C$ depends only on $\ld,\Ld$ , $n$, $\sm$, and $\eta$.
\end{thm}

\pf
(Case {\bf 1}: $1<\sigma_0<\sigma<2$.)

Let $\hat x\in B_{R/2}$ be a point so that
$\inf_{B_{R/2}}u=u(\hat x)$. Then it is enough to show that
$$\sup_{B_{R/2}} u\le C \,\bigl(u(\hat x)+C_0 \bigr).$$ Without loss of
generality, we may assume that $u(\hat x)\le 1$ and $C_0=1$ by
dividing $u$ by $u(\hat x)+C_0 $. Let $\vep_*>0$ be the number given
in Theorem \ref{thm-7.9} and let $\bt=n/\vep_*$. We now set
$s_0=\inf\{s>0:u(x)\le s(1-|x|/R)^{-\bt},\,\forall\,x\in B_R\}.$ Then
we see that $s_0>0$ because $u$ is positive on $\BR^n$. Also there
is some $x_0\in B_R$ such that
$u(x_0)=s_0 (1-|x_0|/R)^{-\bt}=s_0\left(\f{\dd_0}{R}\right)^{-\bt}$ where
$\dd_0=\dd(x_0,\pa B_R)\le R$.

To finish the proof, we have only to show that $s_0$ can not be too
large because $u(x)\le C_1 (1-|x|/R)^{-\bt}\le C$ for any $x\in
B_{R/2}$ if $C_1>0$ is some constant with $s_0\le C_1$. Assume that
$s_0$ is very large. Then by Theorem \ref{thm-7.8} we have that
$$\bigl|\{u\ge u(x_0)/2\}\cap
B_R\}\bigr|\le\biggl|\f{2}{u(x_0)}\biggr|^{\vep_*} |B_R|\le C
s_0^{-\vep_*}\dd_0^n.$$ Since $|B_{r}|=C\dd_0^n$ for $r=\dd_0/2<R$,
we easily obtain that \begin{equation}\label{eq-8.1}\bigl|\{u\ge u(x_0)/2\}\cap
B_{r}(x_0)\}\bigr|\le\biggl|\f{2}{u(x_0)}\biggr|^{\vep_*}\le C
s_0^{-\vep_*}|B_{r}|.\end{equation} In order to get a contradiction, we
estimate $|\{u\le u(x_0)/2\}\cap B_{\dt r}(x_0)|$ for some very
small $\dt>0$ (to be determined later). For any $x\in B_{2\dt
r}(x_0)$, we have that $u(x)\le s_0(\dd_0-\dt\dd_0 /R)^{-\bt}\le
u(x_0)(1-\dt)^{-\bt}$ for $\dt>0$ so that $(1-\dt)^{-\bt}$ is close
to $1$. We consider the function
$$v(x)=(1-\dt)^{-\bt}u(x_0)-u(x).$$ Then we see that $v\ge 0$ on
$B_{2\dt r}(x_0)$, and also $\cM_{\fL_0}^- v\le \f{1}{R^\sigma}$ on
$B_{\dt r}(x_0)$ because $\cM_{\fL_0}^+ u\ge -\f{1}{R^\sigma} $ on
$B_{\dt r}(x_0)$. We now want to apply Theorem \ref{thm-7.9} to $v$.
However $v$ is not positive on $\BR^n$ but only on $B_{\dt r}(x_0)$.
To apply Theorem \ref{thm-7.9}, we consider $w=v^+$ instead of $v$.
Since $w=v+v^-$, we have that $\cM_{\fL_0}^- w\le\cM_{\fL_0}^-
v+\cM_{\fL_0}^+ v^-\le \f{1}{R^\sigma} +\cM_{\fL_0}^+ v^-$ on
$B_{\dt r}(x_0)$. Since $v^-\equiv 0$ on $B_{2\dt r}(x_0)$, if $x\in
B_{\dt r}(x_0)$ then we have that $\mu(v^-,x,y;\n\vp)=v^-(x+y)$,
$y\in B_{\dt r}(x_0)$ and $\vp\in\rC^2_{B_{\dt r}(x_0)}(v^-;x)^+$.
Take any $\vp\in\rC^2_{B_{\dt r}(x_0)}(v^-;x)^+$ and any $x\in
B_{\dt r}(x_0)$. Since $x+B_{\dt r}\subset B_{2\dt r}(x_0)$, we thus
have that
\begin{equation}\label{eq-8.2}
\begin{split}&\cM_{\fL_0}^- w(x;\n\vp)\\&\qquad\le
\f{1}{R^2}+(2-\sm)\int_{\BR^n}\f{\Ld\mu^+(v^-,x,y;\n\vp)-\ld\mu^-(v^-,x,y;\n\vp)}{|y|^{n+\sm}}\,dy\\
&\qquad\le  \f{1}{R^\sigma}+(2-\sm)\int_{\{y\in\BR^n:v(x+y)<0\}}\f{-\Ld \,v(x+y)}{|y|^{n+\sm}}\,dy\\
&\qquad\le  \f{1}{R^\sigma}+(2-\sm)\Ld\int_{\BR^n\s B_{\dt
r}}\f{\bigl(u(x+y)-(1-\dt)^{-\bt}u(x_0)\bigr)_+}{|y|^{n+\sm}}\,dy.\end{split}\end{equation}
We consider the function $h_c(x)=c(1-|x|^2/R^2)_+$ for $c>0$ and we set
$$c_1=\sup\{c>0:u(x)\ge h_c(x),\,\forall\,x\in\BR^n\}.$$ Then
there is some $x_1\in B_R$ such that $u(x_1)=c_1(1-|x_1|^2/R^2)$ and we
see that $c_1\le 4/3$ because $u(\hat x)\le 1$. Since $\n
h_{c_1}(x)=-\f{2 c_1 x}{R^2}$, we have that
\begin{equation}\label{eq-8.3}
\begin{split}
&(2-\sm)\int_{\BR^n}\f{\mu^-(u,x_1,y;\n
  h_{c_1})}{|y|^{n+\sm}}\,dy\\&\qquad\le (2-\sm)\int_{\BR^n}\f{\left(h_{c_1}(x_1+y)
  -h_{c_1}(x_1)-y\cdot \n
  h_{c_1}(x_1) \chi_{B_1}(y)\right)_-}{|y|^{n+\sm}}\,dy\\
&\qquad\le C(2-\sm)\int_{ B_R}\f{|y|^2/R^2}{|y|^{n+\sm}}\,dy+ C(2-\sm)\int_{B_1\s B_R}\f{|y|/R}{|y|^{n+\sm}}\,dy\\
&\qquad\le\f{C(2-\sm_0)}{R^{\sigma}}
\end{split}
\end{equation}
for some constant $C>0$ which is independent of $\sm$, and so we
have that $$\Ld(2-\sm)\int_{\BR^n}\f{\mu^-(u,x_1,y;\n
h_{c_1})}{|y|^{n+\sm}}\,dy\le \f{C}{R^\sigma}.$$ Since
$\cM_{\fL_0}^- u(x_1)\le \f{1}{R^\sigma}$ on $B_{2R}$, by
\eqref{eq-8.3} we have that
\begin{equation*}\begin{split} \f{1}{R^\sigma}&\ge\cM_{\fL_0}^- u(x_1;\n h_{c_1})\\
&\ge\ld(2-\sm)\int_{\BR^n}\f{\mu^+(u,x_1,y;\n
h_{c_1})}{|y|^{n+\sm}}\,dy\\&-\Ld(2-\sm)\int_{\BR^n}\f{\mu^-(u,x_1,y;\n
h_{c_1})}{|y|^{n+\sm}}\,dy.\end{split}\end{equation*} Thus we obtain
that $\,\ds(2-\sm)\int_{\BR^n}\f{\mu^+(u,x_1,y;\n
h_{c_1})}{|y|^{n+\sm}}\,dy\le \f{C}{R^\sigma}\,$ for a constant
$C>0$ which is independent of $\sm$.  We may assume that
$(1-\dt)^{-\bt}u(x_0)=(1-\dt)^{-\bt}s_0(1-|x_0|/R)^{-\bt}\ge 4$
because $s_0$ was very large and $(1-\dt)^{-\bt}$ was close to
$1/R$. Since $\dt r<R$, by the change of variables we have that
\begin{equation*}\begin{split}&(2-\sm)\Ld\int_{B^c_{\dt
r}}\f{\bigl(u(x+y)-(1-\dt)^{-\bt}u(x_0)\bigr)_+}{|y|^{n+\sm}}\,dy\\
&\qquad\le
C(2-\sm)\Ld\int_{\BR^n}\f{\bigl(u(x_1+y)-4/R\bigr)_+}{|y|^{n+\sm}}\,dx\\
&\qquad\leq \,\ds(2-\sm)\int_{\BR^n}\f{\mu^+(u,x_1,y;\n
h_{c_1})}{|y|^{n+\sm}}\,dy \le
\f{C}{R^{\sigma}}\end{split}\end{equation*} for any $x\in B_{\dt
r}(x_0)$. Thus by \eqref{eq-8.2} we obtain that
$$\cM_{\fL_0}^- w(x)\le \f{C}{R^{\sigma}}\leq \f{C}{(\dt r)^{\sigma}}\,\,\text{ on $B_{\dt
r}(x_0)$ }.$$ Since $u(x_0)=s_0(\dd_0/R)^{-\bt}=2^{-\bt}s_0
(r/R)^{-\bt}$ and $\bt\vep_*=n$, applying Theorem \ref{thm-7.9} we
have that
\begin{equation*}\begin{split}&\bigl|\{u\le u(x_0)/2\}\cap B_{\dt r/2}(x_0)\bigr|
=\bigl|\{w\ge u(x_0)((1-\dt)^{-\bt}-1/2)\}
\cap B_{\dt r/2}(x_0)\bigr|\\
&\qquad\le C(\dt r)^n\bigl[((1-\dt)^{-\bt}-1)u(x_0)+C(\dt
r)^{-n-\sm}(\dt
r)^{\sm}\bigr]^{\vep_*}\bigl[u(x_0)((1-\dt)^{-\bt}-1/2)\bigr]^{-\vep_*}\\
&\qquad\le C(\dt r)^n\bigl[((1-\dt)^{-\bt}-1)^{\vep_*}+\dt
^{-n\vep_*}s_0^{-\vep_*}\bigr].\end{split}\end{equation*} We now
choose $\dt>0$ so small enough that $C(\dt
r)^n((1-\dt)^{-\bt}-1)^{\vep_*}\le |B_{\dt r/2}(x_0)|/4.$ Since
$\dt$ was chosen independently of $s_0$, if $s_0$ is large enough
for such fixed $\dt$ then we get that $C(\dt r)^n\dt
^{-n\vep_*}s_0^{-\vep_*}\le |B_{\dt r/2}(x_0)|/4.$ Therefore we
obtain that $\bigl|\{u\le u(x_0)/2\}\cap B_{\dt r/2}(x_0)\bigr|\le
|B_{\dt r/2}(x_0)|/2.$ Thus we conclude that
\begin{equation*}\begin{split}\bigl|\{u\ge u(x_0)/2\}\cap B_r(x_0)\bigr|&\ge\bigl|\{u\ge u(x_0)/2\}\cap
B_{\dt r/2}(x_0)\bigr|\\&\ge\bigl|\{u>u(x_0)/2\}\cap B_{\dt
r/2}(x_0)\bigr|\\&\ge\bigl|B_{\dt r/2}(x_0)\bigr|-\bigl|B_{\dt
r/2}(x_0)\bigr|/2\\&=\bigl|B_{\dt r/2}(x_0)\bigr|/2=C
|B_r|,\end{split}\end{equation*} which contradicts \eqref{eq-8.1} if
$s_0$ is large enough. Thus we complete the proof.

(Case {\bf 2}: $0<\sigma\leq 1$.)
The main step in the argument above is to control the error $\cM_{\fL_0}^+v^-$ for $\cM_{\fL_0,R,\eta}^- w(x;\n\vp)$.
The conclusion comes from the same line of arguments by replacing $\cM_{\fL_0}^-u$ by $\cM_{\fL_0,R,\eta}^- u$.
\qed
\subsection{ H\"older estimates}\label{sec-9} In this subsection, we
obtain H\"older regularity result. The following technical lemma is
very useful in proving it. As in \cite{CS,KL}, its proof can be
derived from Theorem \ref{thm-7.9}.

\begin{lemma}\label{lem-9.1}
Let $\sm_0\in (1,2)$ and assume that  $\,\sm\in
(0,1]$ or $\,\sm\in
(\sm_0,2]$ and that $R\in(0,R_0]$.  If $u$ is a bounded function with $|u|\le 1/2$
on $\BR^n$ such
that
$$\cM_{\fL_0}^- u\le \f{C_0}{R^\sigma}\,\,\,\text{ and }\,\,\,\cM_{\fL_0}^+ u\ge -\f{C_0}{R^\sigma}\,\,\text{
with $\sigma_0<\sigma<2$
on $B_{2R}$}$$
or
$$\cM_{\fL_0,R,\eta}^- u\le \f{C_0}{R^\sigma}\,\,\,\text{ and }\,\,\,\cM_{\fL_0,R,\eta}^+ u\ge -\f{C_0}{R^\sigma}\,\,\text{
with $0<\sigma \leq 1$
on $B_{2R}$}$$

in the viscosity sense where $\vep_0>0$ is some
sufficiently small constant, then there is some universal constant
$\ap>0$ $($depending only on $\ld,\Ld,n$ and $\sm_0$$)$ such that
$u\in\rC^{\ap}$ at the origin. More precisely,
$$|u(x)-u(0)|\le C\,\f{|x|^{\ap}}{R^{\ap}}$$ for some universal constant $C>0$ .
For $\sm\in (\sm_0,2)$, $\alpha$ and $C$ depend only on $\ld,\Ld$ , the dimension $n$, and $\sm_0$. And for $\,\sm\in
(0,1]$ , $\alpha$ and  $C$ depend only on $\ld,\Ld$ , $n$, $\sm$, and $\eta$.
\end{lemma}

Lemma \ref{lem-9.1} and a simple rescaling argument give the
following theorem  as in \cite{CS,KL}.
\begin{thm}\label{thm-9.2}
Let $\sm_0\in (1,2)$ and assume that  $\,\sm\in
(0,1]$ or $\,\sm\in
(\sm_0,2]$ and that $R\in(0,R_0]$.  If $u$ is a bounded function
on $\BR^n$ such
that
$$\cM_{\fL_0}^- u\le \f{C_0}{R^\sigma}\,\,\,\text{ and }\,\,\,\cM_{\fL_0}^+ u\ge -\f{C_0}{R^\sigma}\,\,\text{
with $\sigma_0<\sigma<2$
on $B_{2R}$}$$
or
$$\cM_{\fL_0,R,\eta}^- u\le \f{C_0}{R^\sigma}\,\,\,\text{ and }\,\,\,\cM_{\fL_0,R,\eta}^+ u\ge -\f{C_0}{R^\sigma}\,\,\text{
with $0<\sigma \leq 1$
on $B_{2R}$}$$
in the viscosity sense, then there is some constant
$\ap>0$ such that
$$\|u\|_{\rC^{\ap}(B_{R/2})}\le
\frac{C}{R^{\alpha}}\bigl(\,\|u\|_{L^{\iy}(\BR^n)}+C_0\bigr)$$ where $C>0$ is some
universal constant.
For $\sm\in (\sm_0,2)$, $\alpha$ and $C$ depend only on $\ld,\Ld$ , the dimension $n$, and $\sm_0$. And for $\,\sm\in
(0,1]$ , $\alpha$ and  $C$ depend only on $\ld,\Ld$ , $n$, $\sm$, and $\eta$.
\end{thm}

\subsection{ $\rC^{1,\ap}$-estimates}\label{sec-10}

We consider the class $\fL^1_0$ consisting of the operators
$\cL\in\fL_0$ associated with kernels $K$ for which \eqref{eq-2.3}
holds and there exists some $\vr_1>0$ such that
\begin{equation}\label{eq-10.1}\sup_{h\in B_{\vr_1/2}}\int_{\BR^n\s B_{\vr_1}}\f{|K(y)-K(y-h)|}{|h|}\,dy\le
C.\end{equation}

If $K$ is a radial function satisfying \eqref{eq-2.3}, then it is
interesting that the condition \eqref{eq-10.1} is not required.
Indeed, if $K(y)=(2-\sm)A/|y|^{n+\sm}$ for $\ld\le A\le\Ld$, then it
follows from the mean value theorem and the Schwartz inequality that
\begin{equation*}\begin{split}\sup_{h\in B_{\vr_1/2}}\int_{\BR^n\s
B_{\vr_1}}\f{|K(y)-K(y-h)|}{|h|}\,dy&=(2-\sm)2^{n+\sm+1}\f{(n+\sm)\Ld\om_n}{(\sm+1)\vr_1^{1+\sm}}\le
C\,,\end{split}\end{equation*} because $|y|\ge 2|h|$ for any $h\in
B_{\vr_1/2}$ and $y\in\BR^n\s B_{\vr_1}$ and $|y-\tau h|\ge
|y|-|h|\ge |y|-|y|/2=|y|/2$ for $\tau\in [0,1]$.

If we apply Theorem \ref{thm-9.2} on the H\"older difference
quotients which satisfies the same class of operators as the
solution, we will have the following interior
$C^{1,\alpha}$-estimate as in \cite{CS, CC}. For $R\in(0,R_0]$, we
set $\|u\|^*_{\rC^{1,\ap}(B_{R})}=\|u\|_{L^{\infty}(B_{R})}+R\|D
u\|_{L^{\infty}(B_{R})}+R^{1+\alpha}\|D u\|_{C^{\alpha}(B_{R})}$.

\begin{thm}\label{thm-10.1} For $\sm_0\in (1,2)$, Then there is some $\vr_1>0$ so that if $\,\cI\in \cS^{\fL^1_0}$ for $\sm\in (\sm_0,2)$
or $\,\cI\in \cS^{\fL^1_0}_{\eta}$ for $\sm\in (0,1)$  is a nonlocal
elliptic operator with respect to $\fL^1_0$ in the sense of
Definition \ref{def-3.1} and $u\in\rB(\BR^n)$ is a viscosity
solution to $\cI u=0$ on $B_1$, then there is a universal constant
$\ap>0$ such that
$$\|u\|^*_{\rC^{1,\ap}(B_{R/2})}\le C\bigl(\,\|u\|_{L^{\iy}(\BR^n)}+R^{\sigma}|\cI 0|\,\bigr)$$ for some constant $C>0$
and the constant given in \eqref{eq-10.1} $($where we denote by $\cI
0$ the value we obtain when we apply $\cI$   to the constant
function that is equal to zero$)$. For $\sm\in (\sm_0,2)$, $\rho_1$,
$\alpha$ and $C$ depend only on $\ld,\Ld$ , the dimension $n$, and
$\sm_0$. And for $\sm\in (0,1]$, $\rho_1$, $\alpha$ and  $C\,$
depend only on $\ld,\Ld,$ $n$, $\sm$ and $\eta$.

\end{thm}

\noindent{\bf Acknowledgement.} This work had been started during
Yong-Cheol Kim was visiting to University of Texas at Austin in the fall
semester 2008 for his sabbatical year. He would like to thank for
kind hospitality and deep concern of Professor Luis A. Caffarelli
during that time. We also thank Panki Kim for his helpful discussion. Ki-Ahm Lee was supported by Basic Science Research
Program through the National Research Foundation of Korea(NRF)  grant funded by the Korea government(MEST)(2010-0001985).

\end{document}